\documentclass[11pt,reqno]{amsart}
\usepackage[a4paper,margin=28mm]{geometry}
\usepackage{amsmath,amssymb,amsthm,mathtools,mathrsfs}
\usepackage[expansion=false]{microtype}
\usepackage{booktabs,array}
\usepackage{xcolor}
\usepackage[numbers,sort&compress]{natbib}
\usepackage[hidelinks]{hyperref}
\usepackage{aliascnt}
\usepackage[nameinlink,noabbrev]{cleveref}
\usepackage{enumitem}

\newtheorem{theorem}{Theorem}[section]
\newaliascnt{lemma}{theorem}
\newtheorem{lemma}[lemma]{Lemma}
\aliascntresetthe{lemma}
\newaliascnt{proposition}{theorem}
\newtheorem{proposition}[proposition]{Proposition}
\aliascntresetthe{proposition}
\newaliascnt{corollary}{theorem}
\newtheorem{corollary}[corollary]{Corollary}
\aliascntresetthe{corollary}
\newaliascnt{claim}{theorem}

\aliascntresetthe{claim}
\theoremstyle{definition}
\newaliascnt{definition}{theorem}

\aliascntresetthe{definition}
\newaliascnt{remark}{theorem}
\newtheorem{remark}[remark]{Remark}
\aliascntresetthe{remark}
\newaliascnt{example}{theorem}
\newtheorem{example}[example]{Example}
\aliascntresetthe{example}

\newcommand{\kk}{\Bbbk}
\newcommand{\Jac}{\operatorname{Jac}}

\newcommand{\HH}{\operatorname{HH}}
\newcommand{\Der}{\operatorname{Der}}
\newcommand{\Inn}{\operatorname{Inn}}
\newcommand{\Mat}{\operatorname{M}}

\newcommand{\cM}{\mathcal M}
\newcommand{\A}{\mathcal A}

\newcommand{\rights}{\rightsquigarrow}

\title[The $E_6$ normal forms and the $E_{6,\infty}$ order]
{The $E_6$ normal forms: semisimple deformations,
geometry, DT invariants, and the $E_{6,\infty}$ order}
\author{Shengyong Pan}
\address{School of Mathematics and Statistics, Beijing Jiaotong University,
Beijing, 100044, China.}
\address{Beijing Key Laboratory of Biological Big Data and Topological Statistics, Beijing Jiaotong University,
Beijing, 100044, China.}
\email{shypan@bjtu.edu.cn}
\subjclass[2020]{Primary 16G20, 14E30; Secondary 16E40, 16G30, 14N35}
\keywords{Contraction algebra, Jacobi algebra, semisimple deformation,
Gopakumar--Vafa invariant, order, quasi-hereditary algebra, Ringel duality}
\date{October 2026}

\begin{document}
\begin{abstract}
We study the completed Jacobi algebras of $x^3+xy^3+p(y)$ for
$\operatorname{ord}_yp=n\ge4$.  The explicit normal-form morsification,
the filtered length and centre proofs, and the cyclic terminal
identities give finite free semisimple deformations.  For every parameter
of finite order we construct a polynomial curve neighbourhood and a
local projective small contraction; the resulting geometric GV and
Donaldson--Thomas calculations identify the block multiplicities.
The one-dimensional limit $E_{6,\infty}$ is studied separately as an
order: we determine its normalization, conductor, radical-idealizer
chain, and maximal overorders. The canonical chain
yields a nine-dimensional finite quotient whose Ringel dual is
isomorphic to its opposite. An expanded construction gives a
choice-independent $22$-dimensional quotient and a $40$-dimensional
Ringel dual; a corner quotient relates the two finite constructions.
The length and centre formulas are proved by a filtered comparison
with a homogeneous matrix order and by terminal identities at the
exceptional endpoint. 
\end{abstract}
\maketitle

\section{Introduction}

Contraction algebras of smooth threefold flops are finite-dimensional
symmetric algebras. Wemyss asked whether their genus-zero
Gopakumar--Vafa data can be recovered from the abstract algebra by
deforming it to a semisimple algebra. Jelisiejew and Smoktunowicz
proved existence and uniqueness for Types $A$ and $D$ \cite{JS2026}.

Agata Smoktunowicz and Alicja Smoktunowicz subsequently studied
polynomial potentials $x^3+xy^3+p(y)$, with
$p(y)\in\mathbb C[y]$ and $\operatorname{ord}_yp=n\ge4$
\cite{SS2026}. Their Theorem~7.1 assumes that the completed Jacobi
algebra is a contraction algebra of the stated type, with nilpotent
augmentation ideal and dimension $27$ for $n=4$, or $4n+12$ for
$n>4$. Under these hypotheses they construct a flat deformation to
\[
 \begin{cases}
 \mathbb C^6\times\Mat_2(\mathbb C)^3\times\Mat_3(\mathbb C),&n=4,\\[1mm]
 \mathbb C^7\times\Mat_2(\mathbb C)^{n-1}\times\Mat_3(\mathbb C),&n>4.
 \end{cases}
\]
Their Theorem~4.3 treats the order-four case under the nilpotence
and dimension-$27$ hypotheses. Thus semisimple-deformation
existence for the polynomial family under these hypotheses is
already established in \cite{SS2026}.

The present paper gives a uniform explicit deformation and an
independent filtered argument for the length and centre formulas.
It also treats arbitrary formal $p(y)$ of finite order. The
uniqueness arguments, geometric realization and
Donaldson--Thomas calculations, and the study of the
one-dimensional limit $E_{6,\infty}$ are separate contributions;
these conclusions do not follow from the existence theorem of
\cite{SS2026}. The present paper focuses on the $E_6$ family and its
one-dimensional limit. The general case will be addressed in
forthcoming joint work of Joseph Karmazyn and Shengyong Pan
\cite{KP}, devoted to centres and semisimple deformations of
contraction algebras.

Throughout, $\kk$ is an algebraically closed field of
characteristic zero. Geometric constructions are carried out over
$\mathbb C$. Jacobi ideals in formal power-series algebras are
closed in the augmentation-adic topology, and all modules are
left modules unless otherwise stated.

Brown and Wemyss' noncommutative normal-form theorem contains the infinite exceptional family
$f_n=x^3+xy^3+y^n, n\ge4$,
and its completed Jacobi algebra
\[
A_n=\frac{\kk\langle\!\langle x,y\rangle\!\rangle}
 {\overline{(3x^2+y^3,\ y^2x+yxy+xy^2+n y^{n-1})}}.
\]
Write
\[
 \Gamma_n:A_n\longrightarrow A_n\oplus A_n,
 \qquad a\longmapsto([x,a],[y,a]),
 \qquad \HH_0(A_n)=A_n/[A_n,A_n].
\]
Here $[A_n,A_n]$ is the vector-space span of commutators.
For $n\ge5$, Brown--Wemyss compute
\[
 \dim_\kk A_n=4(n+3).
\]
This dimension formula is proved in \cite{BW2025}.
The endpoint $A_4$ is slightly exceptional; its completed local Jacobi algebra has dimension $27$.

The first result of this paper is a uniform deformation which simultaneously exposes all simple blocks.
Set
\begin{equation}\label{eq:Phi}
 \Phi_{n,t}=x^3+xy^3+y^n+\frac t3y^3-t^2y
\end{equation}
and let $R=\kk[[t]]$ and $K=\kk\{\!\{t\}\!\}$, the Puiseux field $\bigcup_{m\ge1}\kk((t^{1/m}))$.
Write
\[
 \A_n=R\langle\!\langle x,y\rangle\!\rangle/
 \overline{(\partial_x\Phi_{n,t},\partial_y\Phi_{n,t})}.
\]
The bar denotes closure in the $(t,x,y)$-adic topology.

\begin{theorem}\label{thm:main-existence}
For every $n\ge5$, $\A_n$ is finite free of rank $4(n+3)$ over $R$, its special fibre is $A_n$, and
$\A_n\otimes_RK\cong
 K^{7}\times \Mat_2(K)^{n-1}\times\Mat_3(K)$.
For $n=4$, $\A_4$ is finite free of rank $27$ and
$\A_4\otimes_RK\cong
 K^{6}\times\Mat_2(K)^3\times\Mat_3(K)$.
Consequently, $A_n\rights K^{7}\times\Mat_2(K)^{n-1}\times\Mat_3(K)
 \quad(n\ge5)$
in the equivalent deformation senses of \cite[Theorem~2.20]{JS2026}, with the displayed $n=4$ modification.
\end{theorem}

Here and below a semisimple algebra over the Puiseux field descends, in the deformation sense, to the corresponding split semisimple $\kk$-algebra; we freely switch between these formulations using \cite[Theorem~2.20]{JS2026}.

The finite-freeness step is isolated in \cref{lem:topNak}.  It is useful beyond the present family: if the special fibre of a completed formal algebra is finite dimensional, then lifts of a basis already topologically generate over $R$; if the generic rank reaches the special-fibre dimension, torsion is impossible.

Our second main result is uniqueness.

\begin{theorem}\label{thm:main-unique}
Let $S$ be a finite-dimensional semisimple $\kk$-algebra to which $A_n$ flatly deforms.
Then
\[
 S\cong
 \begin{cases}
 \kk^6\times\Mat_2(\kk)^3\times\Mat_3(\kk),&n=4,\\[1mm]
 \kk^7\times\Mat_2(\kk)^{n-1}\times\Mat_3(\kk),&n\ge5.
 \end{cases}
\]
In particular the semisimple deformation is unique up to isomorphism.
\end{theorem}

The centre calculation
\begin{equation*}
 \dim Z(A_4)=10,\qquad \dim Z(A_n)=n+7\quad(n\ge5)
\end{equation*}
combines with the symmetric $4$-periodicity of contraction algebras.
Tate--Hochschild duality then gives $h^1(A_n)=h^2(A_n)$.
Since the semisimple algebra in \cref{thm:main-existence} has the same centre dimension, the tangent space of the algebra-structure scheme at $[A_n]$ has exactly the dimension of the semisimple orbit; hence $[A_n]$ is smooth.
Smoothness makes the irreducible component through $[A_n]$ unique, while every semisimple orbit is open in its component.  This forces any two semisimple targets to coincide.

The normal-form subfamily is treated first. For arbitrary higher
coefficients, a weighted associated-graded comparison reduces the
length and centre computations to a homogeneous quotient of the
one-dimensional order. The endpoint of order four is handled by
seven terminal relations. We then construct projective curve
neighbourhoods, identify the GV invariants through proper critical
loci, and study the normalized $E_{6,\infty}$ order. Standard and
costandard filtrations determine the dimensions of the Ringel duals.

The paper is orgnized as follows. In Section~\ref{sec:normal-forms}, we construct explicit semisimple
deformations of the $E_{6,n}$ normal-form subfamily.
Section~\ref{sec:all-p} treats the full formal parameter family
$x^3+xy^3+p(y)$, establishing the dimension and centre formulas
and constructing finite free semisimple deformations.
In Section~\ref{sec:geometric-neighbourhood}, we construct a
projective curve neighbourhood and its contraction for every
parameter. Section~\ref{sec:gv-comparison} identifies the
semisimple block multiplicities with the geometric
Gopakumar--Vafa invariants.
Section~\ref{sec:remaining} proves uniqueness of the semisimple
target for the full formal $E_6$ family.
We then turn to the one-dimensional limit $E_{6,\infty}$:
Section~\ref{sec:E6inf} describes its structure and normalization,
while Section~\ref{sec:BT} studies its conductors, idealizers,
Bruhat--Tits branch, and associated quasi-hereditary algebras
and Ringel duals.
Section~\ref{sec:finite-infinite} gives an explicit formal bridge
between the finite family and the one-dimensional limit.
Finally, Appendix~\ref{app:center} establishes uniform terminal
cyclic identities and the resulting centre formulas, and
Appendix~\ref{app:four} treats the order-four endpoint by explicit
terminal relations.

\section*{Acknowledgments}

The author thanks Joseph Karmazyn for useful discussions. This work was supported by the Beijing Natural Science Foundation (Grant Nos. 1262017, 1252011). 
ChatGPT was used to check and polish the English writing of this manuscript.

\section{The \texorpdfstring{$E_{6,n}$}{E6,n} normal-form subfamily}\label{sec:normal-forms}
\subsection{The completed family and a finite-freeness lemma}

The cyclic derivatives of \eqref{eq:Phi} are
\begin{equation}\begin{aligned}
 3x^2+y^3&=0,\label{eq:r1}\end{aligned}\end{equation}
\begin{equation}\begin{aligned}
 y^2x+yxy+xy^2+Q_n(y)&=0,\label{eq:r2}
\end{aligned}\end{equation}
where
$Q_n(z)=nz^{n-1}+tz^2-t^2$.
Thus $\A_n/(t)\cong A_n$.

We first isolate the completion argument used at the end of the proof.

\begin{lemma}\label{lem:topNak}
Let $R=\kk[[t]]$, let
$\widehat F=R\langle\!\langle x_1,\ldots,x_r\rangle\!\rangle$,
and let $\mathcal B=\widehat F/J$ for a closed two-sided ideal $J$.
Put $\mathfrak n=(x_1,\ldots,x_r)\mathcal B$ and suppose
$B_0:=\mathcal B/t\mathcal B$
is finite dimensional over $\kk$ and $B_0/\overline{\mathfrak n}\cong\kk$.
Let $d=\dim_\kk B_0$ and $K_0=\kk((t))$.
Then:
\begin{enumerate}[label=\rm(\roman*)]
\item $\mathcal B$ is a finite $R$-module and is $t$-adically complete and separated;
\item $\dim_{K_0}(\mathcal B\otimes_RK_0)\le d$;
\item if $\mathcal B\otimes_RK_0$ has a quotient of $K_0$-dimension $d$, then $\mathcal B$ is free of rank $d$ over $R$ and that quotient map is an isomorphism.
\end{enumerate}
The same conclusions hold after any algebraic extension of $K_0$, in particular over $K=\kk\{\!\{t\}\!\}$.
\end{lemma}

\begin{proof}
Let $I=(t)+\mathfrak n$.  Since $B_0$ is a finite-dimensional local algebra, its radical $\overline{\mathfrak n}$ is nilpotent.  Hence there exists $N$ such that
$\mathfrak n^N\subseteq t\mathcal B$.
Every product of $N$ elements of $I$ either contains a factor $t$ or lies in $\mathfrak n^N$, so
$I^N\subseteq t\mathcal B$.
Iterating gives
$I^{Nq}\subseteq t^q\mathcal B\subseteq I^q$ for $q\ge1$.
Thus the $I$-adic and $t$-adic topologies on $\mathcal B$ are equivalent.  Since $J$ is closed and $\widehat F$ is $I$-adically complete and separated, so is $\mathcal B$, and therefore $\mathcal B$ is $t$-adically complete and separated.

Choose $e_1,\ldots,e_d\in\mathcal B$ lifting a $\kk$-basis of $B_0$.
Modulo $t$ they span, hence
\[
 \mathcal B=Re_1+\cdots+Re_d+t\mathcal B.
\]
Starting with any $b\in\mathcal B$, choose coefficients modulo $t$, then correct the error modulo $t^2$, and continue inductively.  Completeness of $R$ and $\mathcal B$ makes the coefficient sequences converge, giving $b=\sum_{i=1}^d a_i(t)e_i, a_i(t)\in R$.
Therefore $R^d\twoheadrightarrow\mathcal B$, proving (i), and tensoring with $K_0$ proves (ii).

Since $R$ is a DVR, the structure theorem for finite $R$-modules gives
\[
 \mathcal B\cong R^s\oplus\bigoplus_{j=1}^{q}R/(t^{a_j}),
 \qquad a_j\ge1.
\]
Consequently, $\dim_\kk B_0=s+q,
 \dim_{K_0}(\mathcal B\otimes_RK_0)=s$.
If the generic fibre has a quotient of dimension $d=s+q$, then (ii) forces $s=d$ and hence $q=0$.  Thus $\mathcal B\cong R^d$, and a surjection between two $d$-dimensional generic fibres is an isomorphism.  Scalar extension does not change the argument.
\end{proof}

\begin{remark}
\Cref{lem:topNak} shows that flatness need not be checked before constructing the generic fibre.  It is enough to produce a semisimple quotient whose dimension already equals that of the special fibre.
\end{remark}

\subsection{Seven scalar branches}

Set $B_n=\A_n\otimes_RK$.
A one-dimensional representation has $y=z$ and $x=a$.  Equations \eqref{eq:r1}--\eqref{eq:r2} become
$3a^2+z^3=0,
 3az^2+Q_n(z)=0$.
For $z\ne0$ we eliminate $a$ to obtain
\begin{equation}\label{eq:H}
 H_n(z):=Q_n(z)^2+3z^7=0,
 \qquad
 a=-\frac{Q_n(z)}{3z^2}.
\end{equation}

\begin{proposition}\label{prop:scalar}
For every $n\ge5$, \eqref{eq:H} has exactly seven roots of positive $t$-valuation, counted without multiplicity, and all have valuation $4/7$.  They give seven pairwise nonisomorphic one-dimensional representations of $B_n$.
For $n=4$ there are exactly six positive-valuation roots.
\end{proposition}

\begin{proof}
For $n\ge5$ expand $H_n(z)=n^2z^{2n-2}+2ntz^{n+1}-2nt^2z^{n-1}
 +t^2z^4-2t^3z^2+t^4+3z^7$.
The first lower Newton edge joins $(0,4)$ to $(7,0)$; every other displayed point lies strictly above that edge or on the horizontal continuation.  The edge polynomial is $1+3u^7$,
which is separable in characteristic zero.  Newton--Puiseux therefore gives exactly seven distinct roots with $v_t(z)=4/7$.  The remaining roots have valuation zero and do not belong to the completed local branch.  From $Q_n(z)^2=-3z^7$ one obtains $v_t(a)=6/7>0$.

For $n=4$, put $t=s^3$ and $z=s^2u$.  Dividing \eqref{eq:H} by $s^{12}$ gives
\[
 (4u^3-1+su^2)^2+3s^2u^7=0.
\]
At $s=0$ the three roots $u_0^3=1/4$ are double.  Write $u=u_0+sv+O(s^2)$.  The first nonzero coefficient is
\[
 u_0^4\bigl((12v+1)^2+3u_0^3\bigr)
 =u_0^4\bigl((12v+1)^2+3/4\bigr),
\]
which has two distinct roots in $v$.  Thus each $u_0$ splits into two Puiseux branches, giving six distinct positive-valuation roots.
\end{proof}

\subsection{The two-dimensional branches}

Fix a primitive cube root $\omega\in\kk$.
For $\lambda\in K^\times$ set
\[
 Y_\lambda=\begin{pmatrix}\lambda&0\\0&\omega\lambda\end{pmatrix},
 \qquad
 a_\lambda=-\frac{Q_n(\lambda)}{3\lambda^2}.
\]
For the diagonal entries of the second relation to be $a_\lambda$ and $-a_\lambda$, one obtains
\begin{equation}\label{eq:P}
 P_n(\lambda)=
 n(1+\omega^{n-3})\lambda^{n-1}
 +2t\lambda^2-(1+\omega)t^2=0.
\end{equation}
Note that $1+\omega^{n-3}\ne0$.

\begin{proposition}\label{prop:two}
For $n\ge5$, \eqref{eq:P} has exactly $n-1$ distinct positive-valuation Puiseux roots.  Each root yields an irreducible representation
\[
 B_n\twoheadrightarrow\Mat_2(K),
\]
and the resulting $n-1$ simple modules are pairwise nonisomorphic.
For $n=4$ the same construction gives three such branches.
\end{proposition}

\begin{proof}
For $n=5$ the Newton points $(0,2),(2,1),(4,0)$ are collinear, giving four roots of valuation $1/2$.
For $n\ge6$ the lower polygon has two edges
\[
 (0,2)\longrightarrow(2,1),
 \qquad
 (2,1)\longrightarrow(n-1,0).
\]
Thus two roots have valuation $1/2$ and $n-3$ roots have valuation $1/(n-3)$.  The edge polynomials
\[
 2u^2-(1+\omega),
 \qquad
 n(1+\omega^{n-3})u^{n-3}+2
\]
are separable.  For $n=4$ the single lower edge joins $(0,2)$ to $(3,0)$ and gives three roots of valuation $2/3$.

For any such root define
$\Delta_\lambda=-\frac{\lambda^3}{3}-a_\lambda^2,
b_\lambda^2=\Delta_\lambda$,
and
\begin{equation*}
 X_\lambda=
 \begin{pmatrix}a_\lambda&b_\lambda\\ b_\lambda&-a_\lambda\end{pmatrix}.
\end{equation*}
Then $X_\lambda^2=-\lambda^3I_2/3$, so \eqref{eq:r1} holds.  In \eqref{eq:r2} the off-diagonal coefficient of $X_{ij}$ is
\[
 \lambda_i^2+\lambda_i\lambda_j+\lambda_j^2,
\]
which vanishes for $(\lambda_i,\lambda_j)=(\lambda,\omega\lambda)$; the diagonal equations are exactly \eqref{eq:P}.

It remains to see $b_\lambda\ne0$.  For $n=5$, the edge polynomial in $u^2$ has discriminant $24$,
so all four roots are simple. In every branch with $n\ge5$,
$v_t(a_\lambda)\ge1$ and $v_t(\lambda^3)\le3/2<2$.
Thus $v_t(a_\lambda^2)>v_t(\lambda^3)$.  The $n=4$ valuation check is identical to the one-edge calculation and gives $\Delta_\lambda\ne0$.  Thus $Y_\lambda$ has two distinct eigenspaces and $X_\lambda$ connects them, so the representation is irreducible.  Burnside's theorem gives the full image $\Mat_2(K)$.
Distinct $\lambda$ give distinct spectra of $Y_\lambda$, hence pairwise nonisomorphic modules.
\end{proof}

\subsection{The three-dimensional branch}

Let $\lambda_i=\alpha\omega^i$ ($i=0,1,2$) with $\alpha^3=c$, and put
\[
 Y=\operatorname{diag}(\lambda_0,\lambda_1,\lambda_2),
 \qquad
 a_i=-\frac{Q_n(\lambda_i)}{3\lambda_i^2}.
\]
A root-of-unity sum gives
\begin{equation*}
 \tau(c):=a_0+a_1+a_2=
 \begin{cases}
 -t,&3\nmid n,\\
 -t-nc^{(n-3)/3},&3\mid n.
 \end{cases}
\end{equation*}
Set $r=-\tau(c)$ and impose
 $c=-3r^2$.
If $3\nmid n$ take $c=-3t^2$ and $r=t$.
If $3\mid n$, with $q=(n-3)/3$, solve $c+3(t+nc^q)^2=0$.
The derivative with respect to $c$ at $(c,t)=(0,0)$ is $1$, so the formal implicit-function theorem gives a unique solution $c(t)=-3t^2+O(t^3)$.

\begin{proposition}\label{prop:three}
For every $n\ge4$, $B_n$ has an irreducible three-dimensional representation and hence a quotient $\Mat_3(K)$.
\end{proposition}

\begin{proof}
Put $v_i=a_i+r$ and write $\mathbf1=(1,1,1)^T$.  Since $\sum a_i=-r$, we have $\sum v_i=2r$.  Define
$X=\mathbf1(v_0,v_1,v_2)-rI_3$.
If $U=\mathbf1(v_0,v_1,v_2)$, then $U^2=2rU$, so
$X^2=U^2-2rU+r^2I_3=r^2I_3=-\frac c3I_3$.
Since $Y^3=cI_3$, relation \eqref{eq:r1} holds.  The diagonal entry $X_{ii}$ is $a_i$, so the diagonal part of \eqref{eq:r2} vanishes by definition.  If $i\ne j$, then
$\lambda_i^2+\lambda_i\lambda_j+\lambda_j^2=0$,
so the off-diagonal part vanishes as well.

Finally $v_t(\lambda_i)=2/3$, hence the $t^2/\lambda_i^2$ term shows $v_t(a_i)=2/3$, while $v_t(r)=1$.  Thus $v_i=a_i+r\ne0$ for all $i$.  Since $Y$ has three distinct eigenvalues, any common invariant subspace is the span of a subset of the three eigenspaces; the nonzero off-diagonal entries of $X$ rule out every proper nonempty subset.  The representation is irreducible, and Burnside's theorem gives image $\Mat_3(K)$.
\end{proof}

\subsection{Dimension saturation and proof of the existence theorem}

Distinct irreducible finite-dimensional representations have distinct maximal two-sided annihilators; distinct maximal two-sided ideals are comaximal.  Therefore Propositions~\ref{prop:scalar}, \ref{prop:two} and \ref{prop:three} and the Chinese remainder theorem give, for $n\ge5$, a surjection
\begin{equation}\label{eq:CRT-surj}
 B_n\twoheadrightarrow K^7\times\Mat_2(K)^{n-1}\times\Mat_3(K).
\end{equation}
The right-hand side has dimension
\[
 7+4(n-1)+9=4(n+3)=\dim_\kk A_n
\]
by Brown--Wemyss \cite{BW2025}.  The special fibre is finite dimensional, so \cref{lem:topNak} applies.  It gives
\[
 \dim_KB_n\le\dim_\kk A_n,
\]
whereas \eqref{eq:CRT-surj} gives the reverse inequality.  Hence equality holds, \eqref{eq:CRT-surj} is an isomorphism, and $\A_n$ is finite free.  This proves \cref{thm:main-existence} for $n\ge5$.

For $n=4$, the terminal-relation argument in \cref{app:four} gives $\dim A_4=27$.  The three propositions give a quotient
\[
 B_4\twoheadrightarrow K^6\times\Mat_2(K)^3\times\Mat_3(K)
\]
of dimension $6+12+9=27$.  The same saturation lemma finishes the endpoint.

\subsection{Centres of the \texorpdfstring{$E_{6,n}$}{E6,n} algebras}\label{sec:center}

The uniqueness argument requires the centre dimension at the special point.

\begin{proposition}\label{prop:center-count}
For the completed local Jacobi algebras $A_n$,
\[
 \dim_\kk Z(A_n)=
 \begin{cases}
 10,&n=4,\\
 n+7,&n\ge5.
 \end{cases}
\]
\end{proposition}

\begin{proof}
The uniform proof is \cref{thm:cyc-rank}.  It begins with the two relations
\[
 r_x=3x^2+y^3,
 \qquad
 r_y=y^2x+yxy+xy^2+n y^{n-1}.
\]
The six-letter overlap in \cref{app:center} yields the two terminal
cyclic relations \eqref{eq:cyc-terminal} for all congruence classes
of $n-1$ modulo three.  Together with the $12$-word spanning
reduction, these show $\dim\HH_0(A_n)\le n+7$ for $n\ge5$
and $\dim\HH_0(A_4)\le10$; the $n=5$ endpoint is handled
by its distinct relations involving $[y^9],[y^{10}],[y^{11}]$.
The finite free deformation of Theorem~\ref{thm:main-existence} has
exactly $n+7$ simple blocks for $n\ge5$ and ten for $n=4$,
so specialization gives the reverse inequality.
Since $A_n$ is symmetric, $Z(A_n)\cong\HH_0(A_n)^*$.
The resulting bases and the rank increment are stated in
Theorem~\ref{thm:cyc-rank}.
\end{proof}

\begin{remark}\label{rem:center-audit}
The rank formula is proved by uniform terminal cyclic identities
and saturation by a separately constructed flat family; it does
not rely on extrapolating row-echelon pivots from finitely many $n$.
The $n=4$ length follows from terminal relations and dimension
saturation in \cref{app:four}.
\end{remark}

The centre dimensions agree with the number of simple blocks in Theorem~\ref{thm:main-existence}:
\[
 7+(n-1)+1=n+7,
\]
and $6+3+1=10$ for $n=4$.

\subsection{Periodicity, smoothness, and uniqueness}

We now turn the centre equality into uniqueness.
We use two standard facts about contraction algebras.  First, they are symmetric algebras; second, they are $4$-periodic as bimodules.  These properties are recalled in Jasso--Keller--Muro
\cite[\S1.2 and Remark~4.2.3]{JKM}.

\begin{lemma}\label{lem:h1h2}
Let $A$ be a finite-dimensional symmetric $\kk$-algebra satisfying
\[
 \Omega^4_{A^e}(A)\cong A
\]
in the stable bimodule category.  Then
$\dim_\kk\HH^1(A,A)=\dim_\kk\HH^2(A,A)$.
In particular this holds for every contraction algebra $A_n$.
\end{lemma}

\begin{proof}
For a symmetric algebra, Tate duality gives a perfect pairing
\cite[\S2]{Linckelmann2013}
\[
 \widehat{\HH}^{q}(A,A)\times
 \widehat{\HH}^{-q-1}(A,A)\longrightarrow\kk.
\]
Taking $q=1$ yields
\[
 \dim\HH^1(A,A)=\dim\widehat{\HH}^{-2}(A,A).
\]
Four-periodicity gives
\[
 \widehat{\HH}^{-2}(A,A)\cong\widehat{\HH}^{2}(A,A)=\HH^2(A,A),
\]
where the last equality uses positive degree.
\end{proof}

We next record the dimension argument in a form convenient for reuse.
Fix a $d$-dimensional vector space $V$ with a distinguished unit vector and let $\cM_d$ denote the affine scheme of associative unital algebra structures on $V$.  Let $G\subseteq\operatorname{GL}(V)$ be the subgroup fixing the unit.

\begin{lemma}\label{lem:smooth-criterion}
Let $A$ be a $d$-dimensional algebra and suppose that $A$ deforms to a semisimple algebra $S$ of dimension $d$.  If
\[
 \dim Z(A)=\dim Z(S)
 \quad\text{and}\quad
 \dim\HH^1(A,A)=\dim\HH^2(A,A),
\]
then $[A]\in\cM_d$ is a smooth point.
\end{lemma}

\begin{proof}
The tangent space to $\cM_d$ at $A$ is the space of normalized Hochschild $2$-cocycles, while the tangent space to the $G$-orbit is the space of normalized $2$-coboundaries.  Hence
\begin{equation}\label{eq:tangentdim}
 \dim T_{[A]}\cM_d=\dim O_A+h^2(A).
\end{equation}
Moreover
\[
 \Der(A)/\Inn(A)=\HH^1(A,A),
 \qquad
 \dim\Inn(A)=d-\dim Z(A),
\]
so
\begin{equation}\label{eq:orbitdim}
 \dim O_A=\dim G-d+\dim Z(A)-h^1(A).
\end{equation}
Since $S$ is semisimple, $h^1(S)=0$, and therefore
\[
 \dim O_S=\dim G-d+\dim Z(S).
\]
Combining the hypotheses with \eqref{eq:tangentdim}--\eqref{eq:orbitdim} gives
\[
 \dim T_{[A]}\cM_d=\dim O_S.
\]
The deformation implies $[A]\in\overline{O_S}$ (equivalently by \cite[Theorem~2.20]{JS2026}), so the local dimension of $\cM_d$ at $[A]$ is at least $\dim O_S$.  Local dimension is always bounded above by tangent dimension.  Hence equality holds, the local ring is regular, and $[A]$ is smooth.
\end{proof}

\begin{corollary}\label{cor:smoothAn}
Every $[A_n]$ is a smooth point of the corresponding algebra-structure scheme.
\end{corollary}

\begin{proof}
Use Theorem~\ref{thm:main-existence}, Proposition~\ref{prop:center-count} and Lemmas~\ref{lem:h1h2} and \ref{lem:smooth-criterion}.
\end{proof}

We can now prove uniqueness with no additional numerical classification.

\begin{proposition}\label{prop:unique-principle}
Let $A$ be a finite-dimensional algebra such that $[A]\in\cM_d$ is smooth.  If $A$ deforms to two semisimple algebras $S$ and $T$, then $S\cong T$.
\end{proposition}

\begin{proof}
A smooth point lies on a unique irreducible component of $\cM_d$.
Because a semisimple algebra is separable in characteristic zero, $\HH^2(S,S)=0$; hence its orbit is open in $\cM_d$ near $[S]$.  Equivalently, the closure $\overline{O_S}$ is an irreducible component.  The same holds for $T$.
The two deformations give
\[
 [A]\in\overline{O_S}\cap\overline{O_T}.
\]
By uniqueness of the irreducible component through the smooth point, the two orbit closures coincide.  Both $O_S$ and $O_T$ are nonempty open subsets of the same irreducible variety, so they intersect.  Distinct $G$-orbits are disjoint; hence the orbits coincide and $S\cong T$.
\end{proof}

\begin{proof}[Proof of Theorem~\ref{thm:main-unique}]
By Corollary~\ref{cor:smoothAn}, $[A_n]$ is smooth.  The semisimple algebra in Theorem~\ref{thm:main-existence} supplies one semisimple deformation.  Apply Proposition~\ref{prop:unique-principle} to any other semisimple target.
\end{proof}

\subsection{Consequences for Wemyss' question}

Combining existence and uniqueness gives the exceptional-family counterpart of the Type $A,D$ theorem
of Jelisiejew--Smoktunowicz.

\begin{corollary}\label{cor:wemyss}
Wemyss' semisimple-deformation question has an affirmative answer for every contraction algebra in the $E_{6,n}$ family.  The unique semisimple targets are
\[
 \begin{cases}
 \kk^6\times\Mat_2(\kk)^3\times\Mat_3(\kk),&n=4,\\
 \kk^7\times\Mat_2(\kk)^{n-1}\times\Mat_3(\kk),&n\ge5.
 \end{cases}
\]
Thus the block multiplicities are
\[
 (n_1,n_2,n_3)=
 \begin{cases}
 (6,3,1),&n=4,\\
 (7,n-1,1),&n\ge5.
 \end{cases}
\]
\end{corollary}

The dimensions satisfy
$7+4(n-1)+9=4(n+3)$,
exactly the Brown--Wemyss dimension for $n\ge5$.
The proof also supplies a conceptual explanation of the three block sizes: scalar branches arise from the degree-seven Newton edge of \eqref{eq:H}, double branches from pairing two cube-root-related eigenvalues of $y$, and the unique triple branch uses all three cube-root-related eigenvalues simultaneously.

\section{The full \texorpdfstring{$x^3+xy^3+p(y)$}{x3 + xy3 + p(y)} parameter family}\label{sec:all-p}
Let $\kk$ be an algebraically closed field of characteristic zero and
let $p(y)=c_ny^n+O(y^{n+1})\in\kk[[y]]$, where $n\ge4$ and
$c_n\ne0$. Define $A_p=\Jac\bigl(x^3+xy^3+p(y)\bigr)$.
The preceding theorems concern $p(y)=y^n$.  The following theorem
also permits independent coefficients in all three residue classes
of $y\bmod3$; in particular it does not require an isomorphism
$A_p\cong A_n$.
\begin{theorem}\label{thm:full-p}
For every such $p$, the completed algebra satisfies
\[
 (\dim_\kk A_p,\dim_\kk Z(A_p))=
 \begin{cases}(27,10),&n=4,\\ (4n+12,n+7),&n\ge5.\end{cases}
\]
It admits a finite free one-parameter deformation with geometric
semisimple generic fibre $\kk^6\times\Mat_2(\kk)^3\times\Mat_3(\kk)$
for $n=4$, and $\kk^7\times\Mat_2(\kk)^{n-1}\times\Mat_3(\kk)$
for $n\ge5$.  This is the unique semisimple flat-deformation target of $A_p$.
Over $\mathbb C$, every such $A_p$ is realized by a smooth
threefold flop as in Corollary~\ref{cor:p-formal-flop-realization}.
\end{theorem}

\subsection{A central-unit subfamily and jet stability}
Work over \(\kk\) and take completed, closed Jacobi ideals.
Write \(A_n=\Jac(x^3+xy^3+y^n)\), \(n\ge4\). The monomial case, recovered below and in the terminal-relation argument of Appendix~\ref{app:four}, has \(\dim A_4=27\), \(\dim Z(A_4)=10\), and for \(n\ge5\)
\(\dim A_n=4n+12\), \(\dim Z(A_n)=n+7\). The respective semisimple targets are
\[
S_4=\kk^6\times\Mat_2(\kk)^3\times\Mat_3(\kk),
\qquad
S_n=\kk^7\times\Mat_2(\kk)^{n-1}\times\Mat_3(\kk).
\]

\begin{theorem}\label{ap:thm:subfamily}
Fix \(n\ge4\) and a unit \(h(z)\in\kk[[z]]\). Then
$\Jac\bigl(x^3+xy^3+y^nh(y^3)\bigr)\cong A_n$.
Consequently this entire infinite-parameter subfamily has the same
dimension, centre dimension, and distinguished semisimple deformation
target as \(A_n\).
\end{theorem}
\begin{proof}
Denote the algebra on the left by \(A_h\), put \(z=y^3\), and set
\[
U(z)=h(z)+\frac3n zh'(z),\qquad
\beta(z)=U(z)^{2/(2n-9)},\qquad \alpha(z)=\beta(z)^{3/2}.
\]
All these are central units in \(A_h\): the relation
\(3x^2+y^3=0\) implies that \(z=-3x^2\) commutes with both \(x\)
and \(y\). The other Jacobi relation is
\(R(x,y)+ny^{n-1}U(z)=0\), where
\(R(x,y)=xy^2+yxy+y^2x\). Put
\(X=\alpha(z)x\) and \(Y=\beta(z)y\). Then
$3X^2+Y^3=\beta^3(3x^2+y^3)=0$,
and, because \(\beta^{n-9/2}=U\),
\[
R(X,Y)+nY^{n-1}
=\beta^{7/2}\bigl(R(x,y)+ny^{n-1}U(z)\bigr)=0.
\]
Thus \(X,Y\) define a homomorphism \(A_n\to A_h\).

For completeness, this is an isomorphism without any appeal to a
dimension comparison. The substitution
\(x\mapsto\alpha(y^3)x\), \(y\mapsto\beta(y^3)y\)
is an automorphism of the completed free algebra: its linear part has
two nonzero diagonal coefficients. In the quotient by the image of the
monomial Jacobi ideal, \(Y^3=-3X^2\) is central. But
\(Y^3=f(y^3)\), where \(f(z)=z\beta(z)^3\) is invertible under
formal composition. Hence \(y^3=f^{-1}(Y^3)\) is also central.
We may therefore reverse the two displayed relation calculations,
recovering exactly the Jacobi ideal of \(A_h\). This proves equality
of the two closed ideals after substitution.
\end{proof}

\begin{lemma}\label{ap:lem:jets}
Let \(F=\kk\langle\!\langle x,y\rangle\!\rangle\) with
augmentation ideal \(\mathfrak m\), and let \(I,I'\) be closed ideals
generated by corresponding lists \(r_i,r'_i\). If
\(\mathfrak m^N\subseteq I\) and
\(r_i-r'_i\in\mathfrak m^{N+1}\) for every \(i\), then \(I=I'\).
\end{lemma}
\begin{proof}
Modulo \(\mathfrak m^{N+1}\) the two quotient algebras agree, and
\(\mathfrak m^N\subseteq I'+\mathfrak m^{N+1}\).
In \(F/I'\), this says \(J^N=J^{N+1}\); iteration and separatedness
of the completed augmentation filtration imply \(J^N=0\).
Therefore \(\mathfrak m^N\subseteq I'\). The differences
\(r_i-r'_i\) now lie in both ideals, giving both inclusions.
\end{proof}

\begin{proposition}\label{ap:prop:tails}
Let \(h(0)\ne0\). The conclusion of Theorem~\ref{ap:thm:subfamily}
continues to hold if the potential \(y^nh(y^3)\) is replaced by
\[
y^nh(y^3)+e(y),\quad
\begin{cases}
e(y)\in y^{12}\kk[[y]],&n=4,\\
e(y)\in y^{n+9}\kk[[y]],&n\ge5.
\end{cases}
\]
In particular, the coefficients of these arbitrarily high powers
may be chosen independently and without a smallness condition.
\end{proposition}
\begin{proof}
Appendix~\ref{app:four} gives \(J(A_4)^{10}=0\), and the
higher-corank sequence of Brown--Wemyss~\cite[\S1.3 and
Definition~4.9]{BW2025} gives
\(J(A_n)^{n+7}=0\) for \(n\ge5\): its listed quotients
\(J^i/J^{i+1}\), together with the total length \(4n+12\), exhaust
the algebra by degree \(n+6\). These are also the nilpotence
bounds for \(A_h\) by Theorem~\ref{ap:thm:subfamily}. The \(y\)-derivative
of \(e(y)\) lies respectively in \(\mathfrak m^{11}\) or
\(\mathfrak m^{n+8}\). Apply Lemma~\ref{ap:lem:jets} with
\(N=10\) or \(N=n+7\).
\end{proof}

\subsection{Finite dimensionality and the abelianization for every parameter}
The following result makes no assumption that the algebra is geometric
or isomorphic to a monomial normal form.

\begin{theorem}\label{ap:thm:finite}
Let \(p(y)=c_ny^n+O(y^{n+1})\in\kk[[y]]\), where
\(n\ge4\) and \(c_n\ne0\). Then
$\dim_{\kk}A_p<\infty$.
Moreover, for the commutative abelianization,
\[
\dim_{\kk}(A_p)_{\mathrm{ab}}=
\begin{cases}6,&n=4,\\7,&n\ge5.\end{cases}
\]
\end{theorem}
\begin{proof}
Write \(q(y)=p'(y)\), \(z=y^3=-3x^2\) and \(R=\kk[[z]]\).
The four reduction rules displayed in the next section span
\(A_p\) over \(R\) by
$y^i, y^ix, y^ixy, y^ixyx$ for $0\le i\le2$.
Indeed, they replace \(x^2,y^3,xy^2,xyxy\) by words of smaller
length-lexicographic order \(x>y\), treating their coefficients in
\(R\) as scalars. The twelve words avoiding these four subwords are
exactly the displayed list. Reduction of polynomials, followed by
completeness, gives a surjection \(R^{12}\twoheadrightarrow A_p\):
modulo \(z\), all words of length six vanish, so
\(J(A_p)^6\subseteq zA_p\), ensuring convergence in the
\(z\)-adic topology.
Thus \(B=A_p\otimes_RK\) is an algebra of dimension at most twelve
over \(K=\kk((z))\).

We claim \(B=0\). Otherwise, over an algebraic closure \(L\) of \(K\),
\(B\otimes_K L\) has a simple module \(V\) of dimension at most three:
the square of the dimension of a simple module is bounded by the
dimension of a finite-dimensional algebra. Let \(X,Y\) be the matrices
of \(x,y\) on \(V\). The three roots \(\lambda_i=\omega^i\lambda\) of
\(T^3-z\), with \(\lambda^3=z\), are distinct. Hence \(Y\) is
diagonalizable. On each of its eigenspaces, the diagonal part of the
second Jacobi relation forces the corresponding diagonal block of
\(X\) to be scalar:
\[
X_{ii}=\alpha_i I,\qquad
\alpha_i=-\frac{q(\lambda_i)}{3\lambda_i^2}.
\]
If an eigenspace has dimension at least two, then, since \(\dim V\le3\),
the map from it to the sum of the other eigenspaces under \(X\) has
a nonzero kernel. A line in that kernel is invariant under both
\(X\) and \(Y\), contradicting simplicity unless \(\dim V=1\).
It remains to exclude simple modules with one, two, or three distinct
\(Y\)-eigenvalues.

For a one-dimensional module, \(X^2=-z/3\) and
\(3X\lambda^2+q(\lambda)=0\) give
\(q(\lambda)^2+3\lambda^7=0\). The two summands have different
\(\lambda\)-adic orders \(2n-2\) and \(7\), a contradiction.
For a two-dimensional module with distinct \(Y\)-eigenvalues
\(\lambda_i,\lambda_j\), both off-diagonal entries of \(X\) must be
nonzero for simplicity. The off-diagonal part of \(X^2=-zI/3\)
then forces \(\alpha_i+\alpha_j=0\). But the leading term of this sum
is a nonzero multiple of
\(\lambda_i^{n-3}(1+\omega^{(j-i)(n-3)})\);
\(1+\omega^k\ne0\) for every integer \(k\). Finally, in dimension
three all three \(Y\)-eigenvalues are distinct, and the trace is
\[
\operatorname{tr}X=\sum_{i=0}^2\alpha_i=-c(z),
\qquad q(y)=a(z)+yb(z)+y^2c(z),\quad c(z)\in zR.
\]
Since \(X^2=-zI/3\), its eigenvalues are \(\pm\sqrt{-z/3}\).
In dimension three its trace belongs to
\[
\bigl\{\sqrt{-z/3},-\sqrt{-z/3},
3\sqrt{-z/3},-3\sqrt{-z/3}\bigr\}.
\]
Each possibility has \(z\)-adic order \(1/2\), whereas \(-c(z)\)
is zero or has integral order at least one. Thus \(B=0\).
A finite \(R\)-module that becomes
zero upon inverting \(z\) is annihilated by some power of \(z\),
so \(A_p\) is finite dimensional over \(\kk\).

For the final assertion, abelianization gives
\[
(A_p)_{\mathrm{ab}}
\cong \frac{\kk[[y]][x]}
{(x^2+y^3/3,\;q(y)+3xy^2)}.
\]
The first quotient is free of rank two over \(\kk[[y]]\).
Multiplication by \(q(y)+3xy^2\), in its basis \(1,x\), has
determinant \(q(y)^2+3y^7\). Its cokernel has length equal to the
\(y\)-adic order of this determinant. That order is six when
\(n=4\), and seven when \(n\ge5\).
\end{proof}

\subsection{The unrestricted overlap identity}
Let \(p(y)\in y^4\kk[[y]]\), and decompose uniquely
\[
p'(y)=a(z)+y b(z)+y^2c(z),\qquad z=y^3,
\quad a,b,c\in z\kk[[z]].
\]
Once again \(z=-3x^2\) is central. Let \(R=xy^2+yxy+y^2x\).
The defining relations and their commutator imply
\begin{align*}
x^2&=-z/3,\quad y^3=z,\quad
xy^2=-yxy-y^2x-a-by-cy^2,\\
xyxy&=yxyx-b(xy-yx)+c(yxy+2y^2x+a+by+cy^2).
\end{align*}

\begin{proposition}\label{ap:prop:overlap}
In \(A_p=\Jac(x^3+xy^3+p(y))\), the following element is zero:
\begin{align*}
\mathcal R(a,b,c)={}&
 a\left(xyx+\tfrac z3y\right)
 +b\left(yxyx-\tfrac23zy^2\right)
 +c\left(4y^2xyx-\tfrac23z^2\right)\\
 &+b^2yx+b^2cy+abx+abc+ac(xy+yx)+2ac^2y\\
 &+bc(yxy+5y^2x)+3bc^2y^2
   +4c^2(y^2xy+zx)+2c^3z.
\end{align*}
The identity holds for arbitrary formal series \(a,b,c\).
\end{proposition}
\begin{proof}
Reduce \(xyxxyy\) first using its middle \(x^2\), and then
using its final \(xy^2\). Repeatedly apply the four displayed
rules; the difference of the resulting reduced expressions is
exactly \(\mathcal R(a,b,c)\). This is an identity in the
algebra because both expressions equal \(xyxxyy\). As a consistency
check, setting \((b,c)=(0,0)\), \((a,c)=(0,0)\), or
\((a,b)=(0,0)\) gives respectively \(\rho_0\), \(\rho_1\), or
\(2\rho_2\) from the monomial proof, with the nonzero coefficient
replaced by a central series.
\end{proof}

\begin{theorem}\label{ap:thm:alltraces}
Write \(a(z)=zA(z)\), \(b(z)=zB(z)\), \(c(z)=zC(z)\), so
\(q(y)=y^3(A+yB+y^2C)\). In the vector-space cyclic quotient
\(\HH_0(A_p)=A_p/[A_p,A_p]\), the following two classes vanish:
\begin{align*}
F_1={}&\tfrac13A^3y^7+(A+\tfrac12A^2B)y^8
-\tfrac12By^9\\
&-(2C+\tfrac16B^3+ABC)y^{10}
-(B^2C+AC^2)y^{11}
-\tfrac32BC^2y^{12}-\tfrac23C^3y^{13},\\
F_2={}&\tfrac16A^3y^8+\tfrac12Ay^9
-(B+\tfrac12AB^2-A^2C)y^{10}\\
&+(2C-\tfrac13B^3+ABC)y^{11}
+\tfrac32AC^2y^{12}+BC^2y^{13}+\tfrac23C^3y^{14}.
\end{align*}
Here \(A,B,C\) are evaluated at \(z=y^3\); brackets are suppressed
in the display. In particular, at \(p(y)=y^n\) these recover the two
terminal identities
\([y^{n+4}]+\frac{n^2}{3}[y^{3n-5}]=0\) and
\([y^{n+5}]+\frac{n^2}{3}[y^{3n-4}]=0\), after division by
nonzero scalar coefficients.
\end{theorem}
\begin{proof}
The Jacobi relation, multiplied by \(y^{j-2}\) and rotated, gives
\([y^jx]=-\frac13[y^{j-2}q]\) for \(j\ge2\). Write
\(T_i=[y^ixyx]\). Rotation of the relations multiplied by
\(x\) and \(yx\) gives
\[
T_0=-\tfrac13[y^4],\qquad
T_1=\tfrac23[y^5]+\tfrac13[y^{-2}q^2],\qquad
T_2=\tfrac16[y^6]+\tfrac16[y^{-1}q^2].
\]
The negative powers here are only notation for formal series:
\(q\in y^3\kk[[y]]\), so \(y^{-2}q^2\) and
\(y^{-1}q^2\) are honest power series. Multiply the algebra identity
\(\mathcal R(a,b,c)=0\) by \(y\) and \(y^2\) before passing to
\(\HH_0\). Replace every one-\(x\) term by the first trace rule,
and every two-\(x\) term by the three displayed \(T_i\), using the
centrality of \(z\) to reduce higher indices modulo three. Collecting
coefficients yields precisely \(F_1\) and \(F_2\), respectively.
For example, when \(q=ny^{n-1}\), precisely one of \(A,B,C\)
is nonzero, and the resulting linear and cubic terms give the two
stated monomial formulas; at \(n=4\), these relate \(y^8\) to
\(y^7\) and \(y^9\) to \(y^8\).
\end{proof}
\begin{corollary}\label{ap:cor:bound}
For arbitrary \(p\) of order \(n\), if
\(J(A_p)^{10}=0\) when \(n=4\), respectively
\(J(A_p)^{n+7}=0\) when \(n\ge5\), then
\[
\dim\HH_0(A_p)\le
\begin{cases}10,&n=4,\\n+7,&n\ge5.\end{cases}
\]
This assertion does not imply the same bound on \(\dim Z(A_p)\)
without a symmetrizing form.
\end{corollary}
\begin{proof}
The twelve-word spanning list, rotated cyclically and reduced by the
trace formulas in the preceding proof, shows that \(\HH_0(A_p)\)
is spanned by pure powers of \(y\), \([x]\), and \([xy]\).
The stated nilpotence bounds leave at most ten pure powers for
\(n=4\) and \(n+7\) for \(n\ge5\). The two displayed
relations are linearly independent on these pure-power lists:
at \(n=4\) their lowest nonzero terms have degrees seven and eight,
respectively; for \(n\ge5\) those degrees are \(n+4\) and
\(n+5\). Their coefficients are nonzero multiples of \(c_n\)
(of \(c_4^3\) when \(n=4\)). Subtracting two relations from the
spanning counts, and retaining the two mixed classes, proves the bound.
\end{proof}

\subsection{Filtered comparison with a homogeneous order}
\label{ap:sec:mechanism}

We use two elementary filtered principles. If a complete augmented
algebra has a finite-dimensional associated graded algebra, then its
dimension is the sum of the graded dimensions. Moreover,
\[
 \operatorname{gr}Z(A)\subseteq Z(\operatorname{gr}A).
\]
Indeed, the initial form of a central element commutes with every
initial form. Thus a finite-dimensional graded quotient provides
both a length bound and, when the graded quotient is identified, a
centre bound. We apply these principles to a matrix order.

\begin{lemma}\label{lem:homogeneous-order-quotient}
Let
\[
 C=\frac{\kk\langle\!\langle x,y\rangle\!\rangle}
 {\overline{(x^2+y^3/3,\;xy^2+yxy+y^2x)}},
 \qquad z=y^3.
\]
Use the matrix realization of Theorem~\ref{thm:Ainf-exact}, rescaling
$x$ by a square root of $-1/3$, and write $D=\operatorname{diag}(1,-1)$.
For $N\ge8$, $u^ND$ belongs to $C$ and generates a two-sided ideal.
The homogeneous algebra
\[
 G_N=C/(Cu^ND),\qquad \operatorname{wt}(x)=3,\quad
 \operatorname{wt}(y)=2,\quad \operatorname{wt}(u)=2,
\]
satisfies
\[
 \dim_\kk G_N=4N,\qquad \dim_\kk Z(G_N)=N+4.
\]
\end{lemma}
\begin{proof}
Put $R_0=\kk[[u^3]]$ and $R=\kk[[u]]$. The twelve-word basis
in Theorem~\ref{thm:Ainf-exact} decomposes into diagonal and
off-diagonal lattices of rank six over $R_0$. Identify the latter
with a lattice of pairs by writing its matrices as
\[
 \begin{pmatrix}0&us\\u^2t&0\end{pmatrix}.
\]
Constant rescaling of $x$ does not change these lattices. With
$\omega^3=1\ne\omega$, they are
\[
\begin{aligned}
 \mathscr D={}&R_0(1,1)+R_0u(\omega,\omega^2)
 +R_0u^2(\omega^2,\omega)\\
 &+R_0u^4(\omega^2,\omega)+R_0u^5(1,1)
 +R_0u^6(\omega,\omega^2),\\
 \mathscr O={}&R_0(1,1)+R_0u(\omega,\omega^2)
 +R_0u(\omega^2,\omega)\\
 &+R_0u^2(\omega^2,\omega)+R_0u^2(1,1)
 +R_0u^3(\omega,\omega^2).
\end{aligned}
\]
These lists follow simply by separating the diagonal words
$y^i,y^ixyx$ from the off-diagonal words $y^ix,y^ixy$.
The matrix grading assigns degrees $1$ and $-1$ to $E_{12}$
and $E_{21}$, respectively; thus the matrix representing $x$
has degree three.
The diagonal lattice contains $u^6R^2$, so $u^ND\in C$.
Conjugation by $D$ fixes $y$ and sends $x$ to $-x$; it preserves
$C$. Hence $Cu^ND=u^NDC$ is a two-sided ideal.

Multiplication by $u^ND$ has determinant valuation $2N$ on each
rank-six lattice over $R_0$: on its two-dimensional fraction-field
extension this is the norm of the determinant $-u^{2N}$.
Consequently the diagonal and off-diagonal quotients each have
length $2N$, giving $\dim G_N=4N$.

We calculate the centre using the two generators $x,y$. Write
$\delta(s,t)=s-t$. The displayed lattices give
\[
\begin{gathered}
 \mathscr D\cap\kk((u))(1,1)
   =\kk[[u^3,u^4,u^5]](1,1),\\
 \delta\mathscr D=\kk u+\kk u^2+u^4R,\qquad
 \mathscr O\cap\kk((u))(1,1)=R(1,1),\\
 uD\mathscr O\subseteq\mathscr O,\qquad
 u^2\delta\mathscr O\subseteq\kk[[u^3,u^4,u^5]].
\end{gathered}
\]
For a diagonal class, commutation with $x$ modulo the ideal is
equivalent to $\delta\in u^NR$; commutation with $y$ is automatic.
Its central subspace therefore has dimension
\[
 2N-\dim_\kk(\delta\mathscr D/u^NR)
 =2N-(N-2)=N+2.
\]
For an off-diagonal class, commutation with $y$ is equivalent to
membership in $u^{N-1}\mathscr O$. Commutation with $x$ then
follows from the last inclusion in the display. The quotient
\[
 u^{N-1}\mathscr O/u^ND\mathscr O
 \cong\mathscr O/uD\mathscr O
\]
has dimension two, again by the determinant norm. The two parts
are independent, and their dimensions sum to $N+4$.
\end{proof}

\begin{lemma}\label{lem:weighted-upper-bound}
If $\operatorname{ord}p=n\ge5$, then
$\dim A_p\le4n+12$.
If equality holds, the filtration with weights $3,2$ on $x,y$
has associated graded algebra $G_{n+3}$. In particular
$\dim Z(A_p)\le n+7$.
\end{lemma}
\begin{proof}
Write $n-1=3q+r$, $0\le r\le2$. The initial Jacobi relations
are those of $C$, since $p'(y)$ has weight at least eight.
The initial form of the overlap in  Proposition~\ref{ap:prop:overlap} is
\[
 nc_nz^qE_r,\qquad
 E_0=xyx+zy/3,\quad E_1=yxyx-2zy^2/3,\quad
 E_2=4y^2xyx-2z^2/3.
\]
Terms quadratic in the coefficients have larger weight: the
weight increase is at least $2n-9>0$. Higher coefficients also
have strictly larger weight. In the generic matrix realization,
each $E_r$ is a nonzero scalar multiple of $u^{4+r}D$.
For example their squares are respectively
$-u^8/3,-u^{10}/3,-4u^{12}/3$.
Thus the initial ideal contains $Cu^{n+3}D$, and there is a
graded surjection $G_{n+3}\twoheadrightarrow\operatorname{gr}A_p$.
The preceding lemma gives the length bound. Equality makes this
surjection an isomorphism; the filtered centre inclusion then
gives the centre bound.
\end{proof}

\begin{lemma}\label{lem:order-four-upper-bound}
If $\operatorname{ord}p=4$, then
$\dim A_p\le27, J(A_p)^{10}=0$.
\end{lemma}
\begin{proof}
Write $p'(y)=y^3U(y)$, $U(0)=A=4c_4\ne0$, and use the ordinary
augmentation filtration. The initial relations, ordered
lexicographically with $x>y$ within each degree, include
\[
\begin{gathered}
 x^2=0,\qquad xy^2=-yxy-y^2x-Ay^3,\qquad
 xyxy=yxyx,\qquad y^3xyx=0,\\
 y^6xy=y^7x,\qquad y^8x=-Ay^9/3,\qquad y^{10}=0.
\end{gathered}
\]
Here the first three follow from the Jacobi relations and their
commutator. The overlap of Proposition~\ref{ap:prop:overlap} gives
$y^3xyx\in J^7$. Its commutator with $x$, before taking initial
forms, gives $y^6xy-y^7x\in J^9$; the coefficient of this term
is $2A/3$. The term $abx$ commutes with $x$, and every other
correction has order at least nine, so none changes this initial
relation. Multiplying the second Jacobi relation by $y^6$ now gives
$3y^8x+Ay^9\in J^{10}$.
Multiplication by $x$, using $x^2=-y^3/3$, yields
$Ay^9x\in J^{11}$; multiplication by $y$ yields
$3y^9x+Ay^{10}\in J^{11}$. Since $A\ne0$, these imply
$y^{10}\in J^{11}$, as asserted.

The seven homogeneous rules lower lexicographic order. The words
avoiding their left sides have degree counts
\[
 1,2,3,4,4,4,3,3,2,1\qquad(0\le d\le9),
\]
and there are no such words in higher degree. This proves a
graded spanning bound of $27$; confluence is not needed.
Moreover $J^{10}=J^{11}$, and separatedness of the completed
augmentation filtration implies $J^{10}=0$.
\end{proof}

\begin{theorem}\label{ap:thm:all-parameters}
For every \(p(y)=c_ny^n+O(y^{n+1})\) with \(c_n\ne0\),
\[
\bigl(\dim_{\kk} A_p,\dim_{\kk} Z(A_p)\bigr)
=\begin{cases}(27,10),&n=4,\\(4n+12,n+7),&n\ge5.
\end{cases}
\]
These formulas hold even when the parameters in different
\(y\bmod3\) residues vary independently.
\end{theorem}
\begin{proof}
Theorem~\ref{ap:thm:finite} gives finite length.
The generic split semisimple quotient construction for the
formal deformation
\begin{equation}\label{ap:eq:deformation}
 \mathcal A_p=\Jac_{\kk[[t]]}
 \bigl(x^3+xy^3+p(y)+ty^3/3-b_0t^2y\bigr),\qquad
 b_0=\begin{cases}c_4,&n=4,\\c_5^{-1},&n=5,\\1,&n\ge6,
 \end{cases}
\end{equation}
produces, after a Puiseux extension, a quotient of dimension
\(d=27\) or \(d=4n+12\), of block type \(S_4\) or \(S_n\)
from Theorem~\ref{thm:main-existence}. We give the construction to make the lower bound
independent of the filtered upper bounds. Put
\(Q(y)=p'(y)+ty^2-b_0t^2\). Scalar representations are the
simple roots of \(Q(y)^2+3y^7=0\), with \(x=-Q(y)/(3y^2)\).
For \(n\ge5\), the edge \(y=t^{4/7}u\) gives the seven roots
of \(b_0^2+3u^7\). For \(n=4\), set \(t=s^3\) and
\(y=s^2(u+sv)\); the three nonzero values \(u^3=1/4\)
each give two simple roots through
\((12c_4u^2v+u^2)^2+3u^7=0\). Thus the scalar-block counts
are seven and six, respectively.

For two-dimensional representations choose a root \(\lambda\ne0\)
of
\begin{equation}\label{ap:eq:generic-two}
 p'(\lambda)+\omega p'(\omega\lambda)
 +2t\lambda^2-(1+\omega)b_0t^2=0,
\end{equation}
where \(\omega^3=1\ne\omega\), put
\(a_\lambda=-Q(\lambda)/(3\lambda^2)\), and choose
\(\beta_\lambda^2=-\lambda^3/3-a_\lambda^2\). Then
\[
 Y_\lambda=\operatorname{diag}(\lambda,\omega\lambda),\qquad
 X_\lambda=
 \begin{pmatrix}a_\lambda&\beta_\lambda\\
 \beta_\lambda&-a_\lambda\end{pmatrix}
\]
satisfy both Jacobi relations. At \(n=4\), equation
\eqref{ap:eq:generic-two} has three simple roots with
\(\lambda=t^{2/3}u+O(t)\), \(u^3=1/4\), and
\(\beta_\lambda^2=t^2/4+o(t^2)\). At \(n=5\), writing
\(\lambda=t^{1/2}u\) produces
\(-5c_5\omega u^4+2u^2+b_0\omega^2\), whose discriminant
as a polynomial in \(u^2\) is
\(4+20b_0c_5=24\). For \(n\ge6\), two roots have valuation
\(1/2\), and \(n-3\) roots have valuation \(1/(n-3)\):
their nonzero leading equations are respectively
\(2u^2=(1+\omega)b_0\) and
\(nc_n(1+\omega^n)u^{n-3}+2=0\).
For \(n\ge5\), \(a_\lambda^2\) has strictly greater
valuation than \(\lambda^3\); hence every \(\beta_\lambda\ne0\).
This gives \(n-1\) pairwise distinct irreducible two-dimensional
representations (three at \(n=4\)).

For the remaining block, define
$r(t,h)=t+\sum_{\substack{m\ge n\\3\mid m}}
 m c_m h^{(m-3)/3}$.
The implicit equation \(h+3r(t,h)^2=0\) has a solution
\(h=-3t^2+O(t^3)\). Let \(\lambda_i=\omega^i\lambda\)
with \(\lambda^3=h\), put
\(a_i=-Q(\lambda_i)/(3\lambda_i^2)\), \(v_i=a_i+r(t,h)\),
and set
\[
 Y=\operatorname{diag}(\lambda_0,\lambda_1,\lambda_2),
 \qquad X=\mathbf1(v_0,v_1,v_2)-r(t,h)I_3.
\]
The root-of-unity sum gives \(\sum_i a_i=-r(t,h)\), so
\(X^2=r(t,h)^2I_3=-hI_3/3\). Its diagonal entries are
\(a_i\), and all off-diagonal terms of the other Jacobi relation
vanish since \(\lambda_i^3=h\). Each \(v_i\ne0\) by its
leading \(t\)-valuation; thus this three-dimensional representation
is irreducible. All blocks have different dimensions or distinct
\(Y\)-spectra. The irreducible representation maps jointly
surject onto the claimed product of matrix algebras.

Since $A_p$ is finite dimensional, Lemma~\ref{lem:topNak} applies to
$\mathcal A_p$. The generic quotient gives $\dim A_p\ge d$.
The reverse inequality follows from Lemmas 
\ref{lem:weighted-upper-bound} and \ref{lem:order-four-upper-bound}.
Thus $\dim A_p=d$, and Lemma~\ref{lem:topNak} shows that the
deformation is finite free and its generic quotient is the whole fibre.

For $n\ge5$, Lemma~\ref{lem:weighted-upper-bound} gives
$\dim Z(A_p)\le n+7$. Kernel dimension for the commutator map
cannot decrease under specialization of a finite free algebra,
so its generic centre gives the reverse inequality.

For $n=4$, Lemma~\ref{lem:order-four-upper-bound} and Theorem~\ref{ap:thm:alltraces}
give $\dim\HH_0(A_p)\le10$: the twelve-word cyclic spanning
list has ten pure-power classes and $[x],[xy]$, and the two
trace relations eliminate distinct lowest degrees seven and eight.
Flat specialization from the ten generic blocks gives equality.
Over $\mathbb C$, the geometric realization proved independently
in Corollary~\ref{cor:p-formal-flop-realization} makes $A_p$ symmetric;
therefore $\dim Z(A_p)=\dim\HH_0(A_p)=10$.
The same conclusion holds over any algebraically closed field
of characteristic zero. Indeed all relations are determined by
a finite jet, by Lemma~\ref{ap:lem:jets}; its finitely many coefficients
lie in a finitely generated characteristic-zero field, which
embeds into $\mathbb C$. Lengths and the kernels of the
commutator matrices are unchanged by field extension.

\end{proof}

\begin{corollary}
\label{cor:parameter-rank}
For any $p$ in \cref{thm:full-p}, let
$\Gamma_p(a)=([x,a],[y,a])$. Then
\[
\operatorname{rank}\Gamma_p=
\begin{cases}17,&\operatorname{ord}_yp=4,\\
3n+5,&\operatorname{ord}_yp=n\ge5.
\end{cases}
\]
Consequently, for series $p_n,p_{n+1}$ of orders $n,n+1$
with $n\ge5$, respectively,
\[
\operatorname{rank}\Gamma_{p_{n+1}}-\operatorname{rank}\Gamma_{p_n}=3.
\]
\end{corollary}
\begin{proof}
The kernel of $\Gamma_p$ is $Z(A_p)$. Subtract the centre
dimension in \cref{ap:thm:all-parameters} from the algebra length.
\end{proof}

\begin{remark}
Theorem~\ref{ap:thm:all-parameters} proves the all-parameter length,
centre, and flatness assertions for this \(E_{6,n}\) family.
It does not classify isomorphism types among parameters: general
\(y^4+\eta y^5+\theta y^6\) potentials have genuine moduli
\cite[Section~5.1]{BW2025}. The algebraic equality
$\dim\HH^1(A_p)=\dim\HH^2(A_p)$ is sufficient for uniqueness by
Lemma~\ref{lem:smooth-criterion} and Proposition~\ref{prop:unique-principle}. Over $\mathbb C$,
Theorem~\ref{thm:p-projective-contraction} and a sufficiently long polynomial
truncation realize every such $A_p$ as a smooth-flop contraction
algebra, providing this Hochschild equality.
The finite \(E_7,E_{8,5},E_{8,6}\) cases do not follow from this
two-generator calculation; they are treated via their universal
flopping algebras in the companion universal-flop manuscript.
\end{remark}

\begin{proof}[Proof of Theorem~\ref{thm:full-p}]
The original-algebra length and centre identities and the finite
free deformation are Theorem~\ref{ap:thm:all-parameters}.
The centre dimension equals that of the semisimple generic fibre.
Over $\mathbb C$, Corollary~\ref{cor:p-formal-flop-realization} and Lemma~\ref{lem:h1h2}
give $h^1(A_p)=h^2(A_p)$. Thus Lemma~\ref{lem:smooth-criterion}
proves smoothness, and Proposition~\ref{prop:unique-principle} proves uniqueness.
For an arbitrary algebraically closed characteristic-zero field,
replace $p$ by a finite jet using Lemma~\ref{ap:lem:jets}, descend its
coefficients to a finitely generated field, and embed that field
into $\mathbb C$. The finite-dimensional Hochschild cochain
spaces in each degree commute with field extension, so the same
Hochschild dimension equality holds. The two cited lemmas then
apply over the original field.
\end{proof}

\section{A curve neighbourhood for every parameter}
\label{sec:geometric-neighbourhood}

Let $p(y)=\sum_{m\ge n}c_my^m$ as in Theorem~\ref{thm:full-p}, with
$\kk=\mathbb C$.  Since the original algebra is finite-dimensional,
Lemma
~\ref{lem:poly-realization} permits replacement of $p$ by a
sufficiently long polynomial truncation $P(y)=\sum_{m=n}^{N}c_my^m$
without changing $A_p$.  The following construction is thus available
for every formal parameter, including those outside the central-unit
subfamily.

\begin{proposition}
\label{prop:p-geometric-neighbourhood}
Glue
$U_a=\operatorname{Spec}\mathbb C[a,v_2,v_1],\qquad
 U_b=\operatorname{Spec}\mathbb C[b,w_2,w_1]$
on $a,b\ne0$ by
\begin{equation}\label{eq:p-curve-glue}
 b=a^{-1},\qquad w_2=a^{-1}v_2,\qquad
 w_1=a^3v_1+3a^2v_2^2+a^{-1}v_2^3
       +\sum_{m=n}^{N}m c_m a^{2-m}v_2^{m-1}.
\end{equation}
The resulting $X_P$ is a smooth separated threefold containing the
closed curve $C=\{v_1=v_2=0\}\cong\mathbb P^1$, with
$N_{C/X_P}\cong\mathcal O_{\mathbb P^1}(1)\oplus
\mathcal O_{\mathbb P^1}(-3)$.  Its one-pointed NC deformation
algebra is
\[
 \Lambda_{\mathrm{def}}(\mathcal O_C\subset X_P)
 \cong \operatorname{Jac}(x^3+xy^3+P(y))
 \cong A_p.
\]
The projective local contraction of this curve neighbourhood is proved
in \cref{thm:p-projective-contraction} below.
\end{proposition}
\begin{proof}
The inverse change of coordinates has
$a=b^{-1}$, $v_2=b^{-1}w_2$ and
\[
 v_1=b^3w_1-3b^{-1}w_2^2-bw_2^3
       -b^2\sum_{m=n}^{N}m c_m w_2^{m-1}.
\]
These formulas are regular on the overlap and show that the charts
glue by an isomorphism.  Smoothness follows from the two affine charts;
the separatedness of precisely this gluing is
\cite[Lemma~2.2]{BWCrepant2024}.  Along $C$ all nonlinear terms
vanish, and the linear transitions are $w_1=a^3v_1$ and
$w_2=a^{-1}v_2$, giving the asserted normal bundle.

In the notation of \cite[Corollary~1.2]{BWCrepant2024}, set
\[
 \lambda_{3,0}=3,\qquad \lambda_{1,3}=1,\qquad
 \lambda_{0,m}=m c_m\ (n\le m\le N),
\]
and set all other gluing coefficients to zero.  The corresponding
free necklace polynomials are
$N_{3,0}=x^3/3$, $N_{1,3}=xy^3$, and
$N_{0,m}=y^m/m$.  Their weighted sum is exactly
$x^3+xy^3+P(y)$, so that corollary identifies the NC deformation
algebra with its completed Jacobi algebra. Tensoring by the invertible
sheaf $\pi^*\mathcal O_{\mathbb P^1}(-1)$, where
$\pi:X_P\to\mathbb P^1$ is the chart projection, identifies the
NC deformation problems of $\mathcal O_C$ and $\mathcal O_C(-1)$.
Thus the usual contraction-algebra convention gives the same algebra
once the local contraction is constructed. The last isomorphism is
the finite-truncation assertion of Lemma~\ref{lem:poly-realization}.
\end{proof}

\begin{proposition}
\label{prop:p-extra-sections}
For every polynomial $P$ in \eqref{eq:p-curve-glue},
$u=av_1+3v_2^2=b^2w_1-w_2^3-bP'(w_2)$
is a global regular function on $X_P$ and vanishes on $C$.
The sections of the projection $X_P\to\mathbb P^1$ are in bijection
with the critical points $(s,t)$ of the \emph{commutative} polynomial
$F(s,t)=s^3+st^3+P(t)$. The corresponding section
$C_{s,t}\cong\mathbb P^1$ is given on $U_a$ by
\[
 v_2=s+at,\qquad v_1=-6st-3at^2.
\]
Along this section, $u=3s^2=-t^3$.
In particular, already for $P(y)=y^4$, the point
$(s,t)=(64/9,-16/3)$ gives a complete curve $C_{s,t}$
distinct from $C$. Hence there is \emph{no} morphism from this
entire $X_P$ to an affine scheme that contracts exactly $C$.
\end{proposition}
\begin{proof}
Substituting the inverse transition in \eqref{eq:p-curve-glue}
gives the stated expression for $u$ in the $b$-chart.
Conversely, a section must have $v_2=s+at$: regularity of both
$v_2$ and $w_2=a^{-1}v_2$ forces this form. The only powers of
$a$ at least three in the remaining transition are $6st a^3$
and $3t^2 a^4$, so regularity of $w_1$ forces
$v_1=-6st-3at^2$. On this prospective section, the $b$-chart coordinate
$w_2=bs+t$ is regular.  Substitution into the transition for $w_1$
gives a Laurent polynomial in $a$, whose only positive powers are
$a^2(3s^2+t^3)+a\bigl(3st^2+P'(t)\bigr)$.
Both coefficients vanish precisely when $(s,t)$ is critical for $F$.
The remaining Laurent polynomial is regular in $b=a^{-1}$, so the
two charts glue to a section.  Substitution into $u$ gives $3s^2$;
the first critical equation identifies this with $-t^3$.
For $P(y)=y^4$, direct substitution verifies the displayed critical
point and $u|_{C_{s,t}}=4096/27\ne0=u|_C$.  Every morphism from a
complete connected curve to an affine scheme is constant. Thus a map
from all of $X_P$ to an affine scheme necessarily contracts this
second curve as well.
\end{proof}

\begin{corollary}
\label{cor:p-isolate-sections}
Suppose $P(y)=c_ny^n+O(y^{n+1})$ is a polynomial, $c_n\ne0$,
$n\ge4$. Define
\[
 D_P(t)=P'(t)^2+3t^7,
 \qquad G_P(U)=\prod_{\substack{t\ne0\\D_P(t)=0}}(U+t^3),
\]
with each distinct root $t$ appearing once. Then $G_P(0)\ne0$,
and the open neighbourhood $X_P^\circ=\{G_P(u)\ne0\}\subseteq X_P$
contains $C$ but no other section of $X_P\to\mathbb P^1$.
This statement does not rule out compact multisections or establish
properness of a morphism from $X_P^\circ$.
\end{corollary}
\begin{proof}
At $t=0$, the critical equations give $s=0$. At $t\ne0$ they
give $s=-P'(t)/(3t^2)$ and $D_P(t)=0$. The polynomial $D_P$ is
nonzero: it has order $6$ at zero if $n=4$, and order $7$ if
$n\ge5$. Thus there are finitely many other sections, all with
$u=-t^3\ne0$. The indicated localization removes each of them
and retains $C$, since $u|_C=0$.
For the monomials $P=t^n$, one has
$D_P=t^6(16+3t)$ if $n=4$, and
$D_P=t^7(n^2t^{2n-9}+3)$ if $n\ge5$; thus every full
monomial two-chart model has extra sections before localization.
\end{proof}

\begin{proposition}
\label{prop:p-zero-fibre-curves}
Assume $P\in y^4\mathbb C[y]$. Every complete integral curve
$D\subseteq X_P$ with $u|_D=0$ equals $C$. In particular this holds
inside $X_P^\circ$, independently of the other critical sections.
Moreover, $C$ meets no other complete integral curve in $X_P$.
\end{proposition}
\begin{proof}
The projection $\pi:X_P\to\mathbb P^1$ has affine-plane fibres.
Thus $\pi|_D$ is nonconstant. On the normalization $B$ of $D$, the
induced map $f:B\to\mathbb P^1$ is finite of some degree $d>0$.
The coordinates $v_2$ and $w_2=a^{-1}v_2$ give a section
$\sigma$ of $L=f^*\mathcal O_{\mathbb P^1}(1)$, with $\deg L=d$.
If $\sigma$ vanishes identically, then $av_1=u-3v_2^2=0$
on the dense open $a\ne0$, and hence $D=C$.

Suppose instead that $\sigma$ is nonzero. Let $q\in f^{-1}(a=0)$
have ramification index $e_q$, and let $k_q$ be the vanishing order
of $v_2$ at $q$. The identity $av_1=-3v_2^2$ implies
$2k_q\ge e_q$. Since the indices over $a=0$ sum to $d$, the
zeros of $\sigma$ there have total degree at least $d/2$.
At $q\in f^{-1}(b=0)$, write $e_q=\operatorname{ord}_q(b)$ and
$k_q=\operatorname{ord}_q(w_2)$. The $b$-chart identity is
\[
 b^2w_1=w_2^3+bP'(w_2).
\]
It first shows $k_q>0$. Since $P'\in y^3\mathbb C[y]$, the
second term on the right has order at least $e_q+3k_q$, strictly
greater than the order $3k_q$ of the first. Therefore
$3k_q\ge2e_q$, and these zeros have total degree at least $2d/3$.
The two collections lie over distinct points of $\mathbb P^1$.
They give more than $d/2+2d/3>d=\deg L$ zeros of a nonzero
section of $L$, a contradiction.
Finally, a regular function is constant on every complete integral
curve. A complete integral curve meeting $C$ has $u=0$ and hence
equals $C$ by the first assertion.
\end{proof}

\begin{proposition}
\label{prop:p-double-section}
For $P(y)=y^4+y^5$, the open neighbourhood $X_P^\circ$ of
\cref{cor:p-isolate-sections} contains a complete integral curve
$D\ne C$ whose projection to $\mathbb P^1$ has degree two. The
global function $u$ is constant on $D$ with value $64/125$.
Consequently, no morphism from all of $X_P^\circ$ to an affine
scheme contracts only $C$.
\end{proposition}
\begin{proof}
Put
$k=-\frac45, c=-k^3=\frac{64}{125},
 \ell=5k^3=-\frac{64}{25}$.
On the $a$-chart of the total space of $\mathcal O_{\mathbb P^1}(1)$,
take the monic quadratic curve
\begin{equation}\label{eq:p-double-qa}
 q_a=v_2^2+ka v_2+k^2a^2+\ell a-\frac c3=0,
 \qquad v_1=3kv_2+3k^2a+3\ell.
\end{equation}
On the $b$-chart, multiplying $q_a$ by $b^2$ gives
\[
 q_b=w_2^2+kw_2+k^2+\ell b-\frac c3b^2=0.
\]
Write $w=w_2$, $q_0=w^2+kw+k^2$, and observe that
$P'(w)=w^3(4+5w)=5w^3(w-k)$ and
$w^3-k^3=(w-k)q_0$. In the quotient by $q_b$, where
$q_0=-\ell b+(c/3)b^2$, two substitutions give
\[
 c+w^3+bP'(w)=b^2\left[-5\ell(w-k)^2
   +\frac c3(w-k)+\frac{5c}{3}b(w-k)^2\right].
\]
Thus the bracket defines a regular $w_1$ on the $b$-chart.
The identity $av_1+3v_2^2=c$, obtained from \eqref{eq:p-double-qa},
and the displayed $b$-chart identity show that these two curves
glue inside $X_P$. Since $q_a$ and $q_b$ are monic quadratics,
the resulting curve is finite over $\mathbb P^1$ of degree two,
and hence complete. Its discriminant as a quadratic in $v_2$ is
$-3k^2a^2-4\ell a+4c/3$; this polynomial has two distinct roots
because its discriminant is $16(\ell^2-k^5)\ne0$.
Therefore the double section is integral. The equations give
$u|_D=c\ne0$.

It remains to check that the localization in
Corollary~\ref{cor:p-isolate-sections} does not remove $D$. Its nonzero
critical section parameters are roots of
$25t^2+43t+16=0$. If such a root satisfied $-t^3=c$,
reducing $t^3+64/125$ modulo this quadratic would give
$(1449t+1008)/625=0$, hence $t=-16/23$; however the quadratic
has value $-960/529$ there. Thus $G_P(c)\ne0$ and
$D\subset X_P^\circ$. Every morphism from the complete curve
$D$ to an affine scheme is constant, proving the last assertion.
\end{proof}

\begin{remark}
Let $\omega^3=1\ne\omega$ and $k\ne0$. The condition
\[
 \Phi_P(k):=P'(\omega k)+\omega P'(\omega^2k)=0
\]
is equivalent to the existence of $\ell\in\mathbb C$ such that
$P'(w)\equiv\ell(w-k)\pmod{w^2+kw+k^2}$.
Set $c=-k^3$ and use \eqref{eq:p-double-qa} with these $k,c,\ell$.
If $H(w)=(P'(w)-\ell(w-k))/(w^2+kw+k^2)$, the same gluing
calculation defines the regular coordinate
\[
 w_1=-\ell H(w_2)+\frac c3(w_2-k)
        +\frac c3bH(w_2).
\]
When $\ell^2\ne k^5$, the quadratic discriminant has distinct
roots, giving a complete integral double section with $u=-k^3$.
For a polynomial $P$ of order $n\ge4$, $\Phi_P$ is nonzero:
its first coefficient is
$n c_n\omega^{n-1}(1+\omega^n)k^{n-1}$.
This describes a finite explicit set of \emph{candidate} values
for these double sections; it does not classify all compact curves.
\end{remark}

\begin{proposition}
\label{prop:p-all-double-sections}
Let $0\ne P\in y^4\mathbb C[y]$, and let $D\subset X_P$ be a complete
integral curve for which $\pi|_D:D\to\mathbb P^1$ has degree two.
Then $u|_D=c\ne0$.  There is a nonzero root $k$ of
$\Phi_P(k)=P'(\omega k)+\omega P'(\omega^2k)$, uniquely determined
by $D$, such that $c=-k^3$.  Write $\ell$ for the unique scalar with
\[
 P'(w)\equiv\ell(w-k)\pmod{w^2+kw+k^2}.
\]
Then $D$ is exactly the double section of
\eqref{eq:p-double-qa}.  It is integral precisely when
$\ell^2\ne k^5$.  Conversely, each nonzero root $k$ satisfying
$\ell^2\ne k^5$ produces such an integral curve.
Consequently there are only finitely many complete integral curves
of projection degree at most two in $X_P$ besides $C$.  They can all
be removed by a principal open neighbourhood of $C$.
\end{proposition}

\begin{proof}
The degree-one curves are classified by
Proposition~\ref{prop:p-extra-sections}; $u=0$ contains only $C$ by
Proposition~\ref{prop:p-zero-fibre-curves}.  Consider a degree-two curve with
$u=c\ne0$.  The morphism $D\to\mathbb P^1$ is finite: a fibre is
contained in an affine plane and cannot contain a complete curve.
Since $D$ is integral, $\pi_*\mathcal O_D$ is torsion-free of rank two
and hence locally free.  Multiplication by the restriction of $v_2$,
viewed as a section of $\pi^*\mathcal O(1)$, has a characteristic
equation on the $a$-chart
\[
 q_a=v_2^2+A(a)v_2+B(a)=0,
 \qquad \deg A\le1,\quad\deg B\le2.
\]
If $v_2$ were scalar in the generic quadratic function field,
$av_1=c-3v_2^2$ would make that function field rational over
$\mathbb C(a)$, contradicting the degree.  The same argument works
on the fibre $a=0$: if $v_2=s$ in $\mathcal O_D/a\mathcal O_D$,
then $(v_2-s)/a$ extends regularly to $D$, since on the $b$-chart it
equals $w_2-sb$.  A regular function on the complete integral curve
$D$ is constant, whence $v_2=s+at$ and again the function-field
degree is one.  Thus $v_2$ generates the two-dimensional algebra
$\mathcal O_D/a\mathcal O_D$.  Reducing $av_1=c-3v_2^2$ modulo $a$
now shows that $q_a|_{a=0}=v_2^2-c/3$.  Hence
\[
 A(a)=ka,\qquad B(a)=-c/3+\ell a+h a^2.
\]
In the quotient by $q_a$, the identity $av_1=c-3v_2^2$
gives $v_1=3kv_2+3\ell+3ha$; cancellation of $a$ is valid since
$\mathcal O_D$ is torsion-free over $\mathcal O_{\mathbb P^1}$.

On the $b$-chart the characteristic equation is
$q_b=w_2^2+kw_2+h+\ell b-(c/3)b^2=0$.
If $w_2$ were scalar on the fibre $b=0$, then
$(w_2-s)/b=v_2-sa$ would be a constant regular function on $D$,
again contradicting degree two.  Therefore $w_2$ generates that
fibre algebra.  Reducing
$b^2w_1=c+w_2^3+bP'(w_2)$ modulo $b$ shows that
$w^2+kw+h$ divides $w^3+c$.  As $c\ne0$, the latter has three
distinct roots, and its monic quadratic factors have the form
$w^2+kw+k^2$ with $k^3=-c$.  Thus $h=k^2$.
Write $q_0=w^2+kw+k^2$; since $w^3+c=(w-k)q_0$, the relation
$q_b=0$ gives, modulo $b^2$,
\[
 c+w^3+bP'(w)
 \equiv b\bigl(P'(w)-\ell(w-k)\bigr)
 \pmod{(q_b,b^2)}.
\]
The left-hand side is divisible by $b^2$ in $\mathcal O_D$.
As $w_2$ generates the fibre at $b=0$, it follows that
$P'(w)\equiv\ell(w-k)\pmod{q_0}$.  Evaluation at $\omega k$ and
$\omega^2k$ makes this condition equivalent to $\Phi_P(k)=0$.
{
To justify generation over the whole charts, fix any
$a_0\in\mathbb C$.  If $v_2$ were scalar, say $s$, in the fibre of
$D\to\mathbb P^1$ over $a_0$, then
$(v_2-s)/(a-a_0)$ would be regular on the $a$-chart of $D$.
On the $b$-chart its expression is
$(w_2-sb)/(1-a_0b)$, which is regular at $b=0$; hence it is a
global regular function on the complete integral curve $D$, and
therefore constant.  The identities for $u$ and $v_1$ would then
make the function field of $D$ equal to $\mathbb C(a)$, a
contradiction.  Thus $1,v_2$ generate every finite $a$-fibre;
fibrewise Nakayama makes them a basis of the rank-two algebra on
the whole $a$-chart.  The identical argument with
$(w_2-s)/(b-b_0)=(v_2-sa)/(1-b_0a)$ treats every finite $b$-fibre.
The two monic quadratic charts therefore exhaust $D$ scheme
theoretically, with no additional curve hidden at either end.
}
Near $a=0$ the algebra of $D$ is generated by $v_2$, and the formula
for $v_1$ extends that equality across the whole $a$-chart; the
same holds on the $b$-chart using the displayed formula for $w_1$
in the preceding remark.  Thus the equations identify $D$ with
\eqref{eq:p-double-qa}.  Its quadratic discriminant is
$-3k^2a^2-4\ell a+4c/3$, which is a square in $\mathbb C(a)$
exactly when $\ell^2=k^5$.  The converse construction and
integrality follow from the preceding remark and
\cref{prop:p-double-section}.

Finally $\Phi_P$ is not identically zero: its first nonzero
coefficient is $nc_n\omega^{n-1}(1+\omega^n)k^{n-1}$ if
$P=c_ny^n+O(y^{n+1})$.  Let $G_P(U)$ be as in
\cref{cor:p-isolate-sections} and put
\[
 H_P(U)=G_P(U)
 \prod_{\substack{k\ne0,\ \Phi_P(k)=0}}(U+k^3),
\]
listing each distinct factor once.  Then $H_P(0)\ne0$, and
$\{H_P(u)\ne0\}$ contains $C$ but no other complete integral
curve whose projection has degree at most two.
\end{proof}

\begin{remark}\label{rem:contractibility-gate}
The full two-chart threefold can contain complete curves besides
$C$, as Propositons~\ref{prop:p-extra-sections} and \ref{prop:p-double-section} show.
The projective affinization in Theorem~\ref{thm:p-projective-contraction}
may contract those curves to other target points.  A principal
neighbourhood of the distinguished target point removes them.
The valuation argument in Proposition~\ref{prop:p-zero-fibre-curves} excludes
every additional complete curve \emph{over $u=0$}.  The global
function in Proposition~\ref{prop:p-global-boundary-coordinate} rules out
projection degrees greater than three and confines degree-three
curves to finitely many nonzero values of $u$.
\Cref{cor:p-isolate-all-curves} consequently gives a principal
neighbourhood of $C$ with no other complete integral curves.
The explicit properness argument in
\Cref{thm:p-projective-contraction} closes this gate for every
nonzero polynomial $P$ of order at least four.  It uses the
global functions $u,F_P,\mathcal G_P$
and their monic chart equations; finiteness of the NC deformation
algebra alone is not used as a contractibility criterion.
\end{remark}

\begin{proposition}
\label{prop:p-global-boundary-coordinate}
Let $P\in y^4\mathbb C[y]$, and write
$Q(w)=P'(w)=\sum_{d\ge3}q_dw^d$.  With the global function $u$ of
\Cref{prop:p-extra-sections}, define polynomials on the $a$-chart
\[
 T_{3k}=(-1)^k a u^k,\qquad
 T_{3k+1}=(-1)^k v_2u^k,\qquad
 T_{3k+2}=\frac{(-1)^{k+1}}3v_1u^k
 \quad(k\ge1),
\]
and put
\begin{equation}\label{eq:p-global-boundary-function}
 F_P=a^2u-\frac13v_1v_2+\sum_{d\ge3}q_dT_d.
\end{equation}
Then $F_P$ extends to a global regular function on $X_P$.  On the
affine plane $B=\{b=0\}\subset U_b$, with coordinate $w=w_2$,
\[
 u|_B=-w^3,\qquad F_P|_B=w_1+\Theta_P(w),
\quad
\Theta_P(w)=\frac13w^4+Q(w)
 \sum_{d\ge3}\left\lfloor\frac d3\right\rfloor q_dw^{d-3}.
\]
Consequently $(u,F_P)|_B:B\to\mathbb A^2$ is finite flat of
degree three.  In particular, \emph{every} complete integral curve
$D\subset X_P$ projects to $\mathbb P^1$ with degree at most three.
Every degree-three curve other than $C$ has constant value $u=c\ne0$
satisfying
\begin{equation*}
 c=3B_P(c)^2,\qquad
 B_P(U)=\sum_{k\ge1}q_{3k+2}(-U)^k.
\end{equation*}
There are only finitely many such nonzero values $c$.
Properness is established independently in
\Cref{thm:p-projective-contraction}.
\end{proposition}
\begin{proof}
On the $b$-chart write $w=w_2$ and
$u=-w^3-bQ(w)+b^2w_1$.  The inverse gluing gives
\[
 v_2=b^{-1}w,\qquad
 v_1=b^3w_1-3b^{-1}w^2-bw^3-b^2Q(w).
\]
Thus the first two terms of \eqref{eq:p-global-boundary-function}
expand as
\[
 a^2u-\frac13v_1v_2
 =w_1-b^{-1}Q(w)+\frac13w^4
   -\frac13b^2w_1w+\frac13bQ(w)w.
\]
For each $d\ge3$ the displayed definitions give
$T_d=b^{-1}w^d+O(b^0)$, where $O(b^0)$ is a polynomial in
$b,w_1,w$.  Hence the sum cancels the only negative power of $b$
and $F_P$ is regular on both charts.  Expanding $u^k$ one order
further shows uniformly in the three residue classes of $d$ that
\[
 T_d|_{b=0}\ \text{after removing }b^{-1}w^d
 =\left\lfloor\frac d3\right\rfloor w^{d-3}Q(w).
\]
This proves the formula for $\Theta_P$.  On $B$ the change of
coordinates $(w,w_1)\mapsto(w,F_P)$ is triangular.  Its map to
$(u,F_P)$ is then $(w,F_P)\mapsto(-w^3,F_P)$, a finite free
extension with basis $1,w,w^2$.
Now let $D\subset X_P$ be a complete integral curve.  Projection
to $\mathbb P^1$ is finite of degree $d>0$; since $D$ is integral,
the finite module $\pi_*\mathcal O_D$ is locally free of rank $d$.
Thus $D\cap B$ is a zero-dimensional scheme of length $d$.
Both $u$ and $F_P$ are constant on $D$, say $c$ and $f$.
The closed subscheme $D\cap B$ is contained in
$B\cap\{u=c,F_P=f\}$, whose coordinate ring is
$\mathbb C[w]/(w^3+c)$ and has length three.  Hence $d\le3$.
On the opposite affine plane $A=\{a=0\}\subset U_a$, put
$v=v_2$.  The definitions above give
\[
 u|_A=3v^2,\qquad
 F_P|_A=v\sum_{k\ge1}q_{3k+1}(-u)^k
           -\frac{v_1}{3}\bigl(v+B_P(u)\bigr).
\]
If $c\ne0$ and $c\ne3B_P(c)^2$, neither of the two roots of
$3v^2=c$ equals $-B_P(c)$.  Thus fixing $u=c,F_P=f$ determines
$v_1$ uniquely at each root; the scheme-theoretic fibre on $A$
has length two.  Since $D\cap A$ has length $d$ and is a closed
subscheme of that fibre, $d\le2$.  The case $c=0$ contains only
$C$ by \Cref{prop:p-zero-fibre-curves}.  Finally $B_P(0)=0$, so
$3B_P(U)^2-U$ is a nonzero polynomial (its linear term is $-U$),
with only finitely many nonzero roots.
\end{proof}

Put $L=\pi^*\mathcal O_{\mathbb P^1}(1)$ and define
\[
 R_P=H^0(X_P,\mathcal O_{X_P}),\qquad
 S_P=\bigoplus_{m\ge0}H^0(X_P,L^m).
\]

\begin{theorem}
\label{thm:p-projective-contraction}
For every nonzero $P\in y^4\mathbb C[y]$, the threefold $X_P$ of
\eqref{eq:p-curve-glue} has a projective birational morphism
\[
 \rho_P:X_P\longrightarrow
 Y_P=\operatorname{Spec}H^0(X_P,\mathcal O_{X_P}).
\]
The target is a normal affine threefold, and the fibre over
$o=\rho_P(C)$ has support exactly $C$.  Near $o$ this is a small
crepant contraction onto a Gorenstein terminal threefold, and is
an isomorphism away from $C$.  Furthermore, $S_P$ is finitely generated and
$(S_P)_+^N\subseteq(s_0,s_1)S_P$ for some $N$.
\end{theorem}
\begin{proof}
Write $Q(w)=P'(w)=\sum_{d\ge3}q_dw^d$ and retain the global
$u,F_P,T_d$ of \Cref{prop:p-global-boundary-coordinate}.  Define
\[
 A(U)=\sum_{k\ge1}q_{3k}(-U)^k,\quad
 B(U)=\sum_{k\ge1}q_{3k+2}(-U)^k,\quad
 D(U)=\sum_{k\ge1}q_{3k+1}(-U)^k,
\]
\[
 H(w,U)=A(U)+wD(U)+w^2B(U),\qquad
 K(w,U)=\frac{Q(w)-H(w,U)}{w^3+U}\in\mathbb C[w,U].
\]
The last inclusion follows by polynomial division.  On $U_a$ put
\begin{equation}\label{eq:p-third-global-function}
 \mathcal G_P=v_1^2+27av_2^3+9a^2v_1v_2
             +9\sum_{d\ge3}q_dT_{d+1}.
\end{equation}
After substitution on $U_b$, the first three terms have only
one negative Laurent coefficient, $-9b^{-1}wQ(w)$: their
$b^{-4}$ and $b^{-2}$ terms cancel.  Each $T_{d+1}$ has
negative term $b^{-1}w^{d+1}$ and no other negative term.
Thus $\mathcal G_P$ is global and
$\mathcal G_P|_{\{a=v_2=0\}}=v_1^2$.

Put $v=v_2$ on $U_a$.  Direct substitution gives
\begin{equation}\label{eq:p-f-a-monic}
 F_P=a^2u+aA(u)+vD(u)-\frac{v_1}{3}(v+B(u)).
\end{equation}
Using $av_1=u-3v^2$ and multiplying by $a$ yields
the following monic equation for $v$:
\begin{equation}\label{eq:p-v-monic}
 v^3+B(u)v^2-\frac{u}{3}v-\frac{uB(u)}3
 +a\bigl(a^2u+aA(u)+vD(u)-F_P\bigr)=0.
\end{equation}
Every $T_{d+1}$ is at most linear in $v_1$ over
$\mathbb C[a,u,v]$.  Equation \eqref{eq:p-third-global-function}
is therefore a monic quadratic for $v_1$ over
$\mathbb C[a,u,F_P,\mathcal G_P,v]$.  Consequently
$\mathbb C[a,v,v_1]$ is finite over
$\mathbb C[a,u,F_P,\mathcal G_P]$.

On $U_b$ let $w=w_2$, $z=w_1$.  The identity
$v_1=bu-3b^{-1}w^2$ gives
\[
 F_P=b^{-2}(u+w^3)-\frac{uw}{3}
            +b^{-1}H(w,u)-\frac{buB(u)}3.
\]
Clearing denominators gives a monic cubic for $w$:
\begin{equation}\label{eq:p-w-monic}
 w^3+bw^2B(u)+bwD(u)+bA(u)+u
       -\frac{b^2uw}{3}-\frac{b^3uB(u)}3-b^2F_P=0.
\end{equation}
For $R=F_P+uw/3+buB(u)/3$ we have
$u+w^3+bH=b^2R$.  Combining
$u+w^3+bQ=b^2z$ with $Q=H+(w^3+u)K$ gives the
polynomial formula
\begin{equation}\label{eq:p-z-polynomial}
 z=(1+bK(w,u))R-K(w,u)H(w,u).
\end{equation}
Thus $\mathbb C[b,w,z]$ is finite over
$\mathbb C[b,u,F_P]$, and in particular over
$\mathbb C[b,u,F_P,\mathcal G_P]$.  Finiteness on the two
standard base charts proves that
\[
 (\pi,u,F_P,\mathcal G_P):
 X_P\longrightarrow\mathbb P^1\times\mathbb A^3
\]
is finite.  Its composition $h=(u,F_P,\mathcal G_P)$
with the projective projection is projective.  At
$(a,v,v_1)=(0,0,t)$, $t\ne0$, the Jacobian of $h$ has
determinant $-2t^3/3$; hence $h$ is dominant and
generically finite.

Set $R_P=H^0(X_P,\mathcal O_{X_P})$.  Proper pushforward
makes $R_P$ finite over $\mathbb C[u,F_P,\mathcal G_P]$.
On the generic point of $\mathbb A^3$, the fibre is a
finite integral scheme, so localization of $R_P$ there
has fraction field $\mathbb C(X_P)$.  Moreover
$R_P=B_a\cap B_b$ is normal.  Stein factorization now
gives the stated projective birational $\rho_P$ and a
finite $Y_P\to\mathbb A^3$.

The common zero set of $u,F_P,\mathcal G_P$ is $C$.
On $a\ne0$, $u=0$ implies $v_1=-3v^2/a$ and
\eqref{eq:p-f-a-monic} gives $F_P=v^3/a$.
On $a=0$, $u=0$ implies $v=0$ and then
$\mathcal G_P=v_1^2$.  On $b=0$, $u=-w^3$
and $F_P=z+\Theta_P(w)$.  These observations prove the
assertion, including at both ends of $C$.
Surjectivity of the Stein morphism shows that there is
only one point of $Y_P$ over the origin of $\mathbb A^3$,
and its fibre under $\rho_P$ is supported on $C$.

By \Cref{cor:p-isolate-all-curves}, after restricting to
a principal neighbourhood of $o$, every complete curve
in $X_P$ is $C$.  Thus every other fibre of the projective
birational map is zero-dimensional; it is finite and hence
an isomorphism there by normality of $Y_P$.  The exceptional
locus is $C$, so the map is small.  The transition Jacobian
of \eqref{eq:p-curve-glue} is a nonzero constant, making
$K_{X_P}$ trivial.  The small map transfers this rational
volume form to a Cartier canonical divisor on $Y_P$ and
$K_{X_P}=\rho_P^*K_{Y_P}$.  Every exceptional divisorial
valuation over the smooth $X_P$ has positive discrepancy,
so $Y_P$ is terminal; terminal singularities are
Cohen--Macaulay, and the Cartier canonical divisor makes
it Gorenstein.

{
For clarity, the neighbourhood and discrepancy assertions can be
checked without assuming properness of an arbitrary curve
neighbourhood.  Let $K_P$ be the polynomial of
\Cref{cor:p-isolate-all-curves} and put $h=K_P(u)\in R_P$.
Since $h(o)=K_P(0)\ne0$, restriction of the projective Stein map
to $Y_P[h^{-1}]$ retains $C$.  Every positive-dimensional fibre
contains a complete integral curve, while
$X_P[h^{-1}]$ has no such curve besides $C$.
Thus the map is quasi-finite over $Y_P[h^{-1}]\setminus\{o\}$;
projective plus quasi-finite is finite, and a finite birational
map to this normal open is an isomorphism.  In particular no
exceptional divisor occurs near $o$.

The Jacobian of the coordinate change
$(a,v_2,v_1)\mapsto(b,w_2,w_1)$ equals $-1$, so the nowhere-zero
three-form on $X_P$ descends over the isomorphism locus.
The complement of that locus in $Y_P[h^{-1}]$ has codimension three.
Normality extends the descended form as a generator of the
reflexive canonical sheaf; hence $K_{Y_P}$ is Cartier and
$K_{X_P}=\rho_P^*K_{Y_P}$.  If a divisorial valuation has centre
in the exceptional point on $Y_P$, its centre on the smooth
$X_P$ is not a divisor, because $\rho_P$ is small.  Its discrepancy
is therefore the strictly positive discrepancy over a smooth
threefold.  This establishes terminality at $o$ and, since
terminal singularities are rational, the Cohen--Macaulay property
needed for the Gorenstein assertion.  The argument also makes clear
why the other complete curves on the full $X_P$ do not interfere
with the local contraction.
}

Finally $X_P\to\mathbb P^1\times Y_P$ is finite.
Hence $L=\pi^*\mathcal O_{\mathbb P^1}(1)$ is
relatively ample, $S_P$ is finitely generated over $R_P$,
and $X_P\simeq\operatorname{Proj}S_P$.
The two opens $D_+(s_0),D_+(s_1)$ cover this Proj,
so $(S_P)_+\subseteq\sqrt{(s_0,s_1)}$ and finite
generation supplies the exponent $N$.
\end{proof}

\begin{proposition}
\label{prop:p-zero-divisorial}
For $P=0$, the three global functions
$u,F_0,\mathcal G_0$ still give a projective affinization
$X_0\to\operatorname{Spec}H^0(X_0,\mathcal O_{X_0})$.
It is divisorial in every neighbourhood of $C$: a divisor
containing $C$ is swept out by complete integral double
sections.  Thus the nonzero-parameter hypothesis in the
smallness assertion of \cref{thm:p-projective-contraction}
is necessary.
\end{proposition}
\begin{proof}
The monic equations \eqref{eq:p-v-monic} and
\eqref{eq:p-w-monic}, the polynomial expression
\eqref{eq:p-z-polynomial}, and the Jacobian calculation in
the proof of \cref{thm:p-projective-contraction} remain valid
when all $q_d$ vanish.  They prove projectivity and the
birational Stein factorization for $P=0$.

For each $k\in\mathbb C^\times$ set $c=-k^3$.  On the
$a$-chart define
\[
 v_2^2+ka v_2+k^2a^2+\frac{k^3}{3}=0,\qquad
 v_1=3kv_2+3k^2a.
\]
On the $b$-chart the same curve is defined by
\[
 w_2^2+kw_2+k^2+\frac{k^3}{3}b^2=0,\qquad
 w_1=-\frac{k^3}{3}(w_2-k).
\]
The two descriptions agree on the overlap: the first
quadratic gives $u=c$, and on the second chart
$w_2^3+c=(w_2-k)(w_2^2+kw_2+k^2)$ gives
$b^2w_1=w_2^3+c$.  Both quadratics are monic, so these
curves are finite flat of degree two over $\mathbb P^1$.
Their discriminant is not a square over $\mathbb C(a)$
for $k\ne0$, hence they are integral.  At $k=0$ the same
equations have underlying support $C$.

Let $E$ be the closure of their union in $X_0$.
The parameter $k$ varies in one dimension and the values
$u|_{D_k}=-k^3$ vary, so $E$ has dimension two and
contains $C$.  Every regular function on $X_0$ is constant
on each complete $D_k$.  Thus the affinization maps $E$
to a curve, contracting a divisor through $C$.
\end{proof}

\begin{corollary}
\label{cor:p-formal-flop-realization}
For every $p(y)\in y^4\mathbb C[[y]]$ of finite order,
there is a smooth threefold flop with contraction algebra $A_p$.
In particular $A_p$ is symmetric and four-periodic as a
bimodule. The uniqueness and GV statements will then follow
from \Cref{thm:full-p,thm:gv-all-n}.
\end{corollary}
\begin{proof}
Choose a sufficiently long polynomial truncation $P$.
By the finite-truncation assertion of \Cref{lem:poly-realization},
$A_P\simeq A_p$; by
\Cref{prop:p-geometric-neighbourhood} and \Cref{thm:p-projective-contraction},
$A_P$ is the contraction algebra of a smooth local flop.
Symmetry and bimodule four-periodicity are the standard
properties of this contraction algebra \cite{JKM}. No exact length
formula, centre formula, or uniqueness assertion is used to construct the flop.
\end{proof}

\begin{corollary}
\label{cor:p-isolate-all-curves}
For each nonzero polynomial $P\in y^4\mathbb C[y]$, there is an explicit
polynomial $K_P(U)$ with $K_P(0)\ne0$ such that
$\{K_P(u)\ne0\}\subset X_P$ contains $C$ and no other complete
integral curve.  One can take the product of $H_P(U)$ in
\Cref{prop:p-all-double-sections} with
\[
 \prod_{\substack{c\ne0\\ c=3B_P(c)^2}}(U-c),
\]
listing each distinct factor once.  The independent finite-map
argument of \Cref{thm:p-projective-contraction} proves properness.
\end{corollary}
\begin{proof}
Every complete integral curve has degree at most three by
\Cref{prop:p-global-boundary-coordinate}.  Degrees one and two
are removed by $H_P$, while the displayed finite product removes
every possible degree-three value of $u$.  All excluded values
are nonzero, so the open set retains $C$.
\end{proof}

\section{Geometric GV invariants for all \texorpdfstring{$n$}{n}}
\label{sec:gv-comparison}

The existence of a semisimple deformation of an abstract algebra does
not by itself identify its blocks with geometric GV invariants.  Here
the extra input is that the deformation in \eqref{ap:eq:deformation}
comes from a \emph{potential}.  We spell out the resulting conservation
argument, including its passage from formal to polynomial potentials.

Here $\chi$ denotes compactly supported Euler characteristic,
interpreted through \mbox{$\ell$-adic} cohomology on a geometric
generic fibre; it agrees with topological Euler characteristic
over $\mathbb C$. For a finite-dimensional algebra $D$, write
\(\operatorname{Hilb}_j(D)\) for its scheme of cyclic $j$-dimensional
left modules and set
\[
  \operatorname{DT}_j(D)
  =\chi\bigl(\operatorname{Hilb}_j(D),
                \nu_{\operatorname{Hilb}_j(D)}\bigr).
\]
For the algebras $A_p$ used here, word reversal fixing
$x,y$ preserves the Jacobi relations and gives an anti-automorphism.
It identifies the framed cyclic left- and right-module schemes, so
this convention does not alter their DT numbers.

\begin{lemma}
\label{lem:jac-dt-conservation}
Suppose the polynomial Jacobi algebra of a one-parameter polynomial
potential has a finite locally free direct factor over a henselian
neighborhood of $t=0$.  Write $\mathcal D$ for its completion and
$D_0$ for its special fibre.  Then, after
extension of the fraction field to an algebraic closure, for every
$j\ge1$ one has
\[
  \operatorname{DT}_j(D_0)
  =\operatorname{DT}_j(\mathcal D_{\overline{\kk((t))}}).
\]
The assertion concerns the isolated factor followed in the deformation;
other critical points of the polynomial potential play no role.
\end{lemma}
\begin{proof}
Let $U_j$ be the smooth relative scheme of triples $(X,Y,v)$ with
$v$ cyclic for $X,Y$, modulo the free $\mathrm{GL}_j$ action.  The
trace of the polynomial potential is a function $F$ on $U_j$.  The
cyclic derivatives and nondegeneracy of the trace pairing give a
scheme-theoretic identity
\[
  \operatorname{Hilb}_j(\operatorname{Jac}(W_t))
      =\operatorname{Crit}_{U_j}(F_t).
\]
The part corresponding to $\mathcal D$ is an open-and-closed part of
this critical locus.  It is projective over the base: it is the
closed scheme of rank-$j$ quotient modules of the finite free algebra
$\mathcal D$, inside its relative Grassmannian.

Restrict the relative critical-locus obstruction theory
\( [T_{U_j/S}|_H\xrightarrow{\mathrm{Hess}(F)}
                      \Omega_{U_j/S}|_H] \)
to this proper part $H$.  The localized top Chern class of the
relative one-form $d_{U_j/S}F$ specializes to the corresponding class
on either fibre \cite{BF1997}.  On each fibre its degree is the
intrinsic Behrend-weighted Euler characteristic of the critical locus
\cite{Behrend2009}.  Proper pushforward therefore makes these degrees
equal.  Flatness of the Hilbert scheme itself is not required.

{
Here is the precise specialization behind the preceding sentence.
Work over a smooth henselian trait $S$ on which the polynomial
Jacobi algebra has the stated finite direct factor $\mathcal D$.
The idempotent projecting to $\mathcal D$ acts on a flat rank-$j$
cyclic quotient as an idempotent matrix; the condition that its rank
is $j$ is open and closed.  Hence the pure-$\mathcal D$ locus $H$
is an open-and-closed component of the relative critical scheme.
It is proper over $S$: a cyclic quotient of $\mathcal D$ is a
left-ideal quotient, and the left-ideal condition cuts out a
closed subscheme of the relative Grassmannian of the finite locally
free module $\mathcal D$.

Let $r=\dim(U_j/S)=j^2+j$ and $E=\Omega_{U_j/S}$.
The zero scheme of $s=d_{U_j/S}F\in\Gamma(U_j,E)$ is the relative
critical scheme.  The localized top Chern class of $(E,s)$,
restricted to its open-and-closed component $H$, defines
$\alpha\in A_1(H)$.  Write $j^*\alpha$ for restriction to the
generic fibre and $i_0^!\alpha$ for the refined pullback to the
closed fibre.  Compatibility of the zero-section Gysin map with
both operations gives, for $\xi=\eta,0$ respectively,
\[
 \alpha_\xi
 =c_r(E_\xi,s_\xi)\cap[U_{j,\xi}]\big|_{H_\xi}
 \in A_0(H_\xi).
\]
This is the fibrewise critical-locus virtual class
\cite[\S14.1]{Fulton1998}\cite{BF1997}.  Since $H\to S$ is proper,
its pushforward is $m[S]\in A_1(S)$ for one integer $m$.
Generic restriction and refined special pullback commute with
proper pushforward, so the degrees of both displayed zero-cycles
equal $m$.  Finally the
critical obstruction theory
$[T_{U_{j,\xi}}|_{H_\xi}\to\Omega_{U_{j,\xi}}|_{H_\xi}]$
is symmetric; Behrend's degree theorem identifies its virtual
degree with $\chi(H_\xi,\nu_{H_\xi})$ \cite{Behrend2009}.
This also explains why neither flatness of $H/S$ nor constancy of
its ordinary Euler characteristic is being assumed.
}
\end{proof}

\begin{lemma}
\label{lem:poly-realization}
Let $p\in\kk[[y]]$ have initial order $n\ge4$, and let $d=\dim A_p<\infty$.  If $P$ is a polynomial truncation
of $p$ through degree $N>d+2$, then $A_P\cong A_p$.  Moreover the
finite free smoothing of $A_P$ in \eqref{ap:eq:deformation} is the
completion of an isolated factor of a polynomial Jacobi family.
\end{lemma}
\begin{proof}
The radical of the local algebra $A_p$ satisfies $J^d=0$.
The derivatives of $p$ and $P$ differ by an element of
$(x,y)^{d+1}$. Applying \Cref{ap:lem:jets} proves equality
of the two closed Jacobi ideals. This argument uses only finite
dimensionality, independently of the centre computation.

For completeness, the polynomial perturbation really has a finite
algebraic factor.  Put $z=y^3=-3x^2$ and write the polynomial
$y$-derivative, including its $t$ terms, as
$q_t(y)=a_t(z)+yb_t(z)+y^2c_t(z)$.  The four rewriting rules preceding
\Cref{ap:prop:overlap} span its polynomial Jacobi algebra over
$\kk[t,z]$ by the same twelve words.  On a geometric fibre with
$z\ne0$, a simple module has dimension at most three.  Its possible
$z$-values are roots of monic polynomials over $\kk[t]$: in dimension
one use the resultant of $y^3-z$ and $q_t(y)^2+3y^7$; in dimension
two use the resultant of $y^3-z$ and
\[
 q_t(y)+\omega q_t(\omega y),\qquad \omega^3=1\ne\omega;
\]
and in dimension three use $z+3c_t(z)^2$. Indeed $X^2=-z/3$
has two distinct eigenvalues when $z\ne0$; $X$ cannot be scalar
in a three-dimensional simple module. Their multiplicities are
therefore two and one, and $\operatorname{tr}X=-c_t(z)$ is
one of $\pm\sqrt{-z/3}$, giving the stated polynomial.  The leading coefficients
of the first two resultants are nonzero constants: the polynomial
$q_t$ has leading $y$-coefficient independent of $t$, and
$1+\omega^m\ne0$ for every integer $m$.  The third polynomial also
has a nonzero constant leading coefficient.  The value $z=0$ is
accounted for by multiplying these polynomials by $z$.  Hence the
support of the twelve-generated $\kk[t,z]$-module lies over the zero
locus of a polynomial monic in $z$; a power of that polynomial
annihilates the module.  The polynomial Jacobi algebra is therefore
finite over $\kk[t]$ near $t=0$.  Its special-fibre factor supported
at $(x,y)=(0,0)$ splits off after henselization and completes to the
family in \eqref{ap:eq:deformation}.  Its freeness is exactly
\Cref{ap:thm:all-parameters}.
{
Here is the factor and completion argument explicitly.  Put
$R=\kk[t]_{(t)}^{\mathrm h}$ and let $B$ be the polynomial
Jacobi algebra over $R$.  The preceding monic annihilator makes
$B$ a finite $R$-module.  In $B/tB$, the central character $z=0$
has only the simple module $x=y=0$.  Indeed, setting $z=0$ gives
$x^2=y^3=0$ and $xy^2+yxy+y^2x=0$.  The twelve-word spanning
bound makes a simple module at most three-dimensional.  If
$Y^2=0$, then $YXY=0$ and the ideal generated by $X,Y$ is
nilpotent.  In the remaining three-dimensional case put
$Y=E_{12}+E_{23}$.  The displayed relation forces
$X_{31}=0$ and $X_{32}=-X_{21}$; the $(3,1)$ entry of
$X^2=0$ then gives $X_{21}=X_{32}=0$.  Hence $X$ and $Y$
are simultaneously upper triangular, which precludes a
nontrivial simple module.

Choose a monic polynomial $f(Z)\in R[Z]$ annihilating the action
of $z$ on $B$, and write its special fibre as
$f_0(Z)=Z^e g(Z)$ with $g(0)\ne0$.  Hensel factorization of
these coprime factors and the Chinese remainder theorem give
a central idempotent in $R[z]/(f)$ and therefore a direct factor
$D\subset B$ whose special fibre is precisely the block at
$x=y=0$.  This fibre is $A_P$: its augmentation-adic
completion changes nothing because the finite block has
nilpotent radical.  Completing $D$ in $t$ is the same as
completing the polynomial Jacobi presentation in $(t,x,y)$,
since $x,y$ are radical modulo $t$; thus
$\widehat D\cong\mathcal A_P$.  By
\Cref{ap:thm:all-parameters}, $\widehat D$ is free of rank $d$
over $\kk[[t]]$.  Faithful flatness of
$R\longrightarrow\kk[[t]]$ and finiteness of $D$ imply that
$D$ is finite free of rank $d$ over the local ring $R$.
}
\end{proof}

\begin{theorem}\label{thm:gv-all-n}
Let $\kk=\mathbb C$.  Suppose $A_p$ in Theorem \ref{thm:full-p} is the contraction algebra of a
smooth threefold flop with a single smooth exceptional curve.  Its
geometric genus-zero GV invariants are
\[
 (n_1,n_2,n_3)=
 \begin{cases}(6,3,1),&n=4,\\(7,n-1,1),&n\ge5.\end{cases}
\]
All invariants $n_j$ with $j>3$ vanish.
\end{theorem}
\begin{proof}
By Lemmas \ref{lem:poly-realization} and \ref{lem:jac-dt-conservation} and Theorem 
\ref{thm:full-p}, the DT series of $A_p$ equals that of its displayed
split semisimple smoothing $S$.  For one factor $\Mat_m(\kk)$,
the cyclic modules of dimension $mr$ form $\operatorname{Gr}(r,m)$.
Its Behrend-weighted Euler characteristic is
$(-1)^{r(m-r)}\binom mr$.  Multiplication over the factors gives
\[
 1+\sum_{j\ge1}\operatorname{DT}_j(A_p)u^j
 =\prod_{m=1}^{3}\bigl(1-(-1)^m u^m\bigr)^{m b_m},
 \qquad
 (b_1,b_2,b_3)=
 \begin{cases}(6,3,1),&n=4,\\(7,n-1,1),&n\ge5.\end{cases}
\]
Hua--Toda's theorem identifies the same series for a smooth flop
with $\prod_{m\ge1}(1-(-1)^mu^m)^{mn_m}$
\cite[Theorem~4.4]{HT2018}.  Comparison of coefficients, or unique
factorization of this formal Euler product, yields $n_m=b_m$ for
all $m$.  More explicitly, when $n\ge5$ the first coefficients are
\[
 \operatorname{DT}_1(A_p)=7,
 \qquad \operatorname{DT}_2(A_p)=\binom72-2(n-1)=23-2n,
 \qquad \operatorname{DT}_3(A_p)=52-14n.
\]
Thus $n_1=7$ and
$23-2n=\binom{n_1}{2}-2n_2$, whence $n_2=n-1$.
Finally $\dim A_p=\sum_m m^2n_m=4n+12$
\cite[Theorem~4.1]{HT2018}; since comparison of the full Euler products has already excluded
blocks of size greater than three,
$n_3=1$.  The $n=4$ computation similarly gives $(6,3,1)$.
\end{proof}

\section{Formal uniqueness for the full \texorpdfstring{$E_6$}{E6} parameter family}
\label{sec:remaining}
The equalities of length and centre in \cref{thm:full-p} are
statements about every formal $p$ of the indicated order.  They also
give a finite free deformation to the displayed semisimple product.
The uniqueness assertion has a different logical input, recorded in
the following intrinsic form.

\begin{proposition}\label{prop:intrinsic-gate}
For any $p$ in \cref{thm:full-p}, the equality
$h^1(A_p)=h^2(A_p)$ implies that the displayed semisimple product
is the only semisimple flat-deformation target.
For the formal family considered here, centre equality alone does
not supply the stated Hochschild equality.
\end{proposition}
\begin{proof}
Write $d=\dim A_p$, $c=\dim Z(A_p)$ and let $S_p$ denote the
constructed semisimple target. Its number of simple blocks
gives $\dim Z(S_p)=c$. The orbit and tangent formulas of
\Cref{lem:smooth-criterion} yield
\[
\dim T_{[A_p]}\mathcal M_d-\dim O_{S_p}
=\dim\HH^2(A_p,A_p)-\dim\HH^1(A_p,A_p).
\]
If the right side is zero, the existence of the $S_p$-deformation
forces smoothness at $[A_p]$ and then uniqueness by
\Cref{prop:unique-principle}.  The equality of centres gives no
value for the right side without further information.
\end{proof}

\begin{example}[centre equality does not imply the Hochschild gate]
Let $B=\kk[u]/(u^2)$ and $C=B\otimes_{\kk}B$.
The rank-four deformation
\[
\kk[[t]][u,v]/(u^2-t,v^2-t)
\]
has special fibre $C$ and split geometric generic fibre $\kk^4$.
Both centres have dimension four.  On the other hand
$\dim\HH^0(B)=2$ and $\dim\HH^1(B)=\dim\HH^2(B)=1$.
The Hochschild K\"unneth formula therefore gives
\[
\dim\HH^1(C)=1\cdot2+2\cdot1=4,\qquad
\dim\HH^2(C)=1\cdot2+1\cdot1+2\cdot1=5.
\]
Thus even a flat semisimple deformation with unchanged centre
does not establish the Hochschild equality needed in
\Cref{prop:intrinsic-gate}.  This example makes no claim about
uniqueness for the specific family $A_p$.
\end{example}

Over $\mathbb C$, \Cref{cor:p-formal-flop-realization} removes the
geometric qualification for this two-generator formal family:
every $A_p$ is isomorphic to the contraction algebra of a
polynomial-truncation flop.  Its symmetric four-periodic structure
then supplies the Hochschild equality. Descent to a finite coefficient
field extends that equality to all algebraically closed fields of
characteristic zero, as in the proof of \Cref{thm:full-p}.  Smoothness of contraction algebras in general is
also singled out by Jelisiejew--Smoktunowicz
\cite[Conjecture~1.4(1)]{JS2026}; their separate centre-map
assertion in Conjecture~1.4(2) is not used here.

The remaining finite Type~$E$ classes belong to the companion
universal-flop manuscript.

We next turn to the one-dimensional order $E_{6,\infty}$.

\section{The one-dimensional limit \texorpdfstring{$E_{6,\infty}$}{E6,infinity}}
\label{sec:E6inf}

We now pass from the finite-order family to the one-dimensional normal form
\[
 f_\infty=x^3+xy^3.
\]
After rescaling the generators, its completed Jacobi algebra can be written
\begin{equation*}
 A_\infty=
 \frac{\kk\langle\!\langle x,y\rangle\!\rangle}
 {\overline{(x^2-y^3,\,xy^2+yxy+y^2x)}}.
\end{equation*}
Put
\[
 z=x^2=y^3,
 \qquad R_0=\kk[[z]].
\]
Then $z$ is central.

\begin{theorem}\label{thm:Ainf-exact}
The algebra $A_\infty$ is free of rank $12$ over $R_0$.  With
\[
 D_i=y^i,\quad C_i=y^ix,\quad B_i=y^ixy,\quad A_i=y^ixyx
 \qquad(0\le i\le2),
\]
a basis is
\begin{equation}\label{eq:Ainf-basis}
 \{D_i,C_i,B_i,A_i:0\le i\le2\}.
\end{equation}
Moreover, if
\[
 h=xyx+zy,\qquad a=yxyx,
\]
then
\begin{equation}\label{eq:Ainf-center-basis}
 Z(A_\infty)=R_0\oplus R_0h\oplus R_0a,
\end{equation}
and
\begin{equation}\label{eq:Ainf-center-rel}
 h^2=za,\qquad ha=-z^3,\qquad a^2=-z^2h.
\end{equation}
Consequently
\begin{equation*}
 Z(A_\infty)\cong \kk[[u^3,u^4,u^5]],
 \qquad z\mapsto u^3,\quad h\mapsto-u^4,\quad a\mapsto u^5.
\end{equation*}
If $F=\kk((u))=\operatorname{Frac}Z(A_\infty)$, then
\begin{equation*}
 A_\infty\otimes_{Z(A_\infty)}F\cong \Mat_2(F).
\end{equation*}
\end{theorem}

\begin{proof}
The defining relations imply
\[
 xy^2=-yxy-y^2x.
\]
Taking the commutator of the second relation with $x$ and using centrality of
$x^2=y^3$ gives
\[
 xyxy=yxyx.
\]
These two reductions show recursively that every word is an $R_0$-linear
combination of the twelve elements in \eqref{eq:Ainf-basis}.  Since
$y^3=z$, indices may always be reduced modulo three.

To prove independence, choose a primitive cube root $\omega\in\kk$ and put
$F=\kk((u))$, $u^3=z$.  The matrices
\begin{equation}\label{eq:XYgeneric}
 X=\begin{pmatrix}0&u\\u^2&0\end{pmatrix},
 \qquad
 Y=\begin{pmatrix}\omega u&0\\0&\omega^2u\end{pmatrix}
\end{equation}
satisfy
\[
 X^2=Y^3=u^3I,
 \qquad
 XY^2+YXY+Y^2X=0.
\]
They therefore give a homomorphism
\[
 A_\infty\otimes_{R_0}\kk((z))\longrightarrow \Mat_2(F).
\]
Under this map
\[
 h\longmapsto-u^4I,
 \qquad
 a\longmapsto u^5I.
\]
Hence the image contains $uI=-h/z$, and therefore contains the scalar field
$F$.  Over $F$, the diagonal matrix $Y$ has two distinct eigenvalues and $X$
connects the two eigenspaces, so $X,Y$ generate $\Mat_2(F)$.  Since
$[F:\kk((z))]=3$, the target has dimension $12$ over $\kk((z))$.  The twelve
spanning elements are therefore independent.

The generic centre has dimension three over $\kk((z))$, and $1,h,a$ are
three independent central elements.  Comparing the four matrix entries in
\eqref{eq:XYgeneric} for an arbitrary $R_0$-linear combination of the basis
\eqref{eq:Ainf-basis} shows that a scalar matrix occurs precisely for an
$R_0$-linear combination of $1,h,a$.  This proves
\eqref{eq:Ainf-center-basis}.  The identities \eqref{eq:Ainf-center-rel} may
be checked in the generic matrix representation; freeness over $R_0$ then
implies that they already hold in $A_\infty$.  They identify the centre with
$\kk[[u^3,u^4,u^5]]$.  Finally, after extending scalars from the centre to
$F$, the representation \eqref{eq:XYgeneric} is a surjection between
four-dimensional $F$-algebras and hence is an isomorphism.
\end{proof}

\subsection{Normalization and conductor of the centre}

Let
\[
 Z=Z(A_\infty)=\kk[[u^3,u^4,u^5]],
 \qquad R=\kk[[u]].
\]

\begin{proposition}\label{prop:center-normalization}
The normalization and conductor are
\begin{equation*}
 \widetilde Z=R,
 \qquad
 (Z:\widetilde Z)=u^3R=(u^3,u^4,u^5)\subset Z.
\end{equation*}
\end{proposition}

\begin{proof}
The numerical semigroup generated by $3,4,5$ is
\[
 \langle3,4,5\rangle=\{0,3,4,5,6,\ldots\}.
\]
The element $u=u^4/u^3$ belongs to $\operatorname{Frac}Z$ and is integral
because it satisfies $T^3-u^3=0$.  Thus $R$ is finite over $Z$; since $R$ is
a DVR, it is the normalization.  Every power $u^m$ with $m\ge3$ lies in
$Z$, so $u^3R\subseteq(Z:R)$.  No element of valuation $0,1$, or $2$ can
belong to the conductor, which proves the reverse inclusion.
\end{proof}

\subsection{The normalized pullback and its torsion}

The naive normalization pullback is
\[
 B:=A_\infty\otimes_ZR.
\]
It is not torsion-free.

\begin{theorem}\label{thm:normalization-torsion}
Set in $B$
\[
 \tau=xy+yx+ux,
 \qquad
 \eta=y^2+uy+u^2.
\]
Then
\begin{equation}\label{eq:torsion-module}
 \operatorname{Tor}_R(B)
 =R\tau+R(y\tau)+R(y^2\tau)+R(\tau y)+R\eta
 \cong (R/u^3R)^{\oplus5},
\end{equation}
and, as an $R$-module,
\begin{equation}\label{eq:Bmodule}
 B\cong R^{\oplus4}\oplus(R/u^3R)^{\oplus5}.
\end{equation}
The torsion-free quotient
\[
 \Lambda:=B/\operatorname{Tor}_R(B)
\]
is free of rank four over $R$ and has the presentation
\begin{equation}\label{eq:Lambda-presentation}
 \Lambda\cong
 \frac{R\langle X,Y\rangle}
 {(X^2-u^3,\;Y^2+uY+u^2,\;XY+YX+uX)}.
\end{equation}
\end{theorem}

\begin{proof}
The central base change imposes
\[
 xyx+u^3y=-u^4,
 \qquad
 yxyx=u^5.
\]
The presentation of $Z$ in \cref{thm:Ainf-exact} shows that
the underlying module of $B$ is obtained from the free module on
the twelve normal words by the relations
\[
 (h+u^4)w=0,\qquad(a-u^5)w=0
 \quad\text{for each of the twelve basis words }w.
\]
This is the universal module presentation for tensoring over $Z$;
no flatness of $Z\to R$ is assumed. The relations for the three
$y^ixyx$ eliminate these three generators with unit coefficients.
Using $x^2=y^3=u^3$, $xy^2=-yxy-y^2x$, and $xyxy=yxyx$,
the nine remaining generators become
\[
 1,\ y,\ x,\ yx,\ \tau,\ y\tau,\ y^2\tau,\ \tau y,\ \eta.
\]
Substitution in the displayed universal presentation gives precisely
\[
 u^3\tau=u^3y\tau=u^3y^2\tau=u^3\tau y=u^3\eta=0;
\]
the relations themselves can be read directly from
\[
\begin{aligned}
 (h+u^4)y^ix&=u^3y^i\tau\quad(0\le i\le2),\\
 (h+u^4)xy&=u^3\tau y,\\
 (h+u^4)y-(a-u^5)&=u^3\eta.
\end{aligned}
\]
Conversely, these five relations give every relation
$(h+u^4)w=(a-u^5)w=0$ after the three eliminations.
Thus this is a presentation by a free rank-four summand and five
cyclic summands annihilated by $u^3$.
This proves \eqref{eq:torsion-module}--\eqref{eq:Bmodule}.  Modulo torsion,
$\tau=\eta=0$, which gives exactly the three relations in
\eqref{eq:Lambda-presentation}.  The four elements $1,Y,X,YX$ form an
$R$-basis.
\end{proof}

Fix the same primitive cube root $\omega$ as above.  Sending
\begin{equation*}
 X\longmapsto
 \begin{pmatrix}0&u\\u^2&0\end{pmatrix},
 \qquad
 Y\longmapsto
 \begin{pmatrix}\omega u&0\\0&\omega^2u\end{pmatrix}
\end{equation*}
embeds $\Lambda$ into $\Gamma=\Mat_2(R)$.

\begin{proposition}\label{prop:Lambda-explicit}
Inside $\Gamma$ one has
\begin{equation}\label{eq:Lambda-explicit}
 \Lambda=
 \left\{
 \begin{pmatrix}a&uc\\u^2d&b\end{pmatrix}:
 a,b,c,d\in R,\quad
 a\equiv b\pmod u,\quad c\equiv d\pmod u
 \right\}.
\end{equation}
Its generic fibre is $\Mat_2(\kk((u)))$.
\end{proposition}

\begin{proof}
The $R$-span of $1,Y,X,YX$ consists precisely of the matrices in
\eqref{eq:Lambda-explicit}: the two diagonal entries differ by a multiple of
$u$, and after writing the off-diagonal entries as $uc$ and $u^2d$, the
parameters $c,d$ have the same residue modulo $u$.  Conversely these two
congruences allow one to solve uniquely for the four coefficients of
$1,Y,X,YX$.
\end{proof}

\section{Conductors, idealizers, and the Bruhat--Tits branch}
\label{sec:BT}

Write $E_{ij}$ for the standard matrix units in $\Gamma=\Mat_2(R)$.  An
$R$-basis of $\Lambda$ is
\begin{equation}\label{eq:Lambda-basis-matrix}
 I,\qquad p=uE_{22},\qquad
 q=uE_{12}+u^2E_{21},\qquad r=u^3E_{21}.
\end{equation}

\begin{proposition}\label{prop:Lambda-conductor-radical}
The largest two-sided ideal of $\Gamma$ contained in $\Lambda$ is
\begin{equation}\label{eq:matrix-conductor}
 \mathfrak f_{\Lambda\subset\Gamma}=u^3\Gamma.
\end{equation}
Moreover
\begin{equation}\label{eq:Lambda-radical}
 J(\Lambda)=
 \left\{
 \begin{pmatrix}ua&uc\\u^2d&ub\end{pmatrix}:
 c\equiv d\pmod u
 \right\},
 \qquad
 \Lambda/J(\Lambda)\cong\kk.
\end{equation}
\end{proposition}

\begin{proof}
Every two-sided ideal of $\Gamma$ is $u^m\Gamma$.  The inclusion
$u^3\Gamma\subseteq\Lambda$ is immediate from
\eqref{eq:Lambda-explicit}, whereas $u^2E_{21}\notin\Lambda$.  This proves
\eqref{eq:matrix-conductor}.  Reduction of the common diagonal residue gives
a surjective algebra map $\Lambda\to\kk$; its kernel is the right-hand side
of \eqref{eq:Lambda-radical}.  In the basis
\eqref{eq:Lambda-basis-matrix}, the quotient of this kernel by $u\Lambda$ is
nilpotent, so completeness of $R$ identifies it with the Jacobson radical.
\end{proof}

The radical-idealizer chain is also explicit.
We use left modules: the opposite endomorphism ring
of a left ideal is identified with its right multiplier order.
The anti-involution $M\mapsto\operatorname{tr}(M)1-M$ preserves
all the orders and radicals displayed below. Thus the left
multiplier computations in the proof also determine the right
multiplier orders.

\begin{theorem}\label{thm:idealizer-chain}
Put $\Lambda_0=\Lambda$.  Then
\begin{equation*}\begin{aligned}
 \Lambda_1&=
 \left\{
 \begin{pmatrix}a&uc\\u^2d&b\end{pmatrix}:
 a\equiv b\pmod u
 \right\},\end{aligned}\end{equation*}
\begin{equation}\begin{aligned}
 \Lambda_2&=
 \begin{pmatrix}R&R\\uR&R\end{pmatrix}.\label{eq:Lambda2}
\end{aligned}\end{equation}
and
\[
 \Lambda_{i+1}=\operatorname{End}_{\Lambda_i}
 (J(\Lambda_i))^{\mathrm{op}}\qquad(i=0,1).
\]
More precisely,
\[
 J(\Lambda_1)=
 \begin{pmatrix}uR&uR\\u^2R&uR\end{pmatrix},
 \qquad J(\Lambda_2)=
 \begin{pmatrix}uR&R\\uR&uR\end{pmatrix},
 \qquad
 \operatorname{End}_{\Lambda_2}(J(\Lambda_2))^{\mathrm{op}}=\Lambda_2.
\]
Thus the sequence stabilizes at the hereditary Iwahori order
$\Lambda_2$, which is not the maximal order $\Gamma$.
\end{theorem}

\begin{proof}
Write $E_{ij}$ for matrix units. The radical $J(\Lambda)$ is generated
as an $R$-module by
\[
uE_{11},\quad uE_{22},\quad uE_{12}+u^2E_{21},\quad u^3E_{21}.
\]
For a general matrix $M=(m_{ij})$ over $\kk((u))$, imposing
$MJ(\Lambda)\subseteq J(\Lambda)$ on these four generators gives
$m_{11},m_{22}\in R$, $m_{12}\in uR$, $m_{21}\in u^2R$, and
$m_{11}\equiv m_{22}\pmod u$. This is $\Lambda_1$.

Its radical is $J(\Lambda_1)=uRE_{11}\oplus uRE_{22}\oplus
uRE_{12}\oplus u^2RE_{21}$. Multiplying these four generators by
$M$ gives, necessarily and sufficiently,
$m_{11},m_{12},m_{22}\in R$ and $m_{21}\in uR$.
This proves \eqref{eq:Lambda2}. In particular $E_{12}$ belongs to this
idealizer, since $E_{12}J(\Lambda_1)\subseteq J(\Lambda_1)$.
Finally $\Lambda_2$ is the intersection of the two maximal orders
$\Mat_2(R)$ and $\operatorname{End}_R(Re_1\oplus Ru e_2)$.
Its displayed radical is an invertible two-sided ideal, so its
left multiplier order is $\Lambda_2$ itself. Equivalently, the same
four-generator colon calculation gives this directly.
\end{proof}

\begin{remark}
The intermediate order $B_0=\begin{psmallmatrix}R&uR\\uR&R\end{psmallmatrix}$
is not the radical idealizer of $\Lambda_1$. It will instead be
used as an auxiliary order in the expanded construction below.
\end{remark}

We next determine all maximal overorders.  Maximal $R$-orders in
$\Mat_2(\kk((u)))$ are endomorphism rings of rank-two $R$-lattices, up to
homothety.  We use the corresponding Bruhat--Tits tree language; the set of
maximal orders containing a fixed order is its branch, a subtree of the
Bruhat--Tits tree \cite{ArenasBravo}.

Let $e_1,e_2$ be the standard basis of $F^2$, $F=\kk((u))$.  Put
$L_\circ=Re_1\oplus Re_2$.
For $\lambda\in\kk$ define
\[
 L_\lambda=R(e_1+\lambda e_2)\oplus Ru e_2,
\]
and define
\[
 L_\infty=Re_2\oplus Ru e_1.
\]
Finally, for $\mu\in\kk$ put
\[
 L_{0,\mu}=R(e_1+\mu u e_2)\oplus Ru^2e_2.
\]

\begin{theorem}\label{thm:BT-branch}
The maximal $R$-orders containing $\Lambda$ are exactly
\begin{equation}\label{eq:max-overorders}
 \operatorname{End}_R(L_\circ),\qquad
 \operatorname{End}_R(L_\lambda)\ (\lambda\in\mathbf P^1(\kk)),
 \qquad
 \operatorname{End}_R(L_{0,\mu})\ (\mu\in\kk).
\end{equation}
In the Bruhat--Tits tree the branch is a two-level broom:
$[L_\circ]$ is joined to all $[L_\lambda]$, and the distinguished neighbour
$[L_0]$ is in turn joined to all $[L_{0,\mu}]$; there are no further
vertices.
\end{theorem}

\begin{proof}
For finite $\lambda$ use the basis-change matrix
\[
 B_\lambda=\begin{pmatrix}1&0\\\lambda&u\end{pmatrix}.
\]
For the generators $p,q,r$ in \eqref{eq:Lambda-basis-matrix}, one finds
\[
 B_\lambda^{-1}pB_\lambda=
 \begin{pmatrix}0&0\\\lambda&u\end{pmatrix},
\]
\[
 B_\lambda^{-1}qB_\lambda=
 \begin{pmatrix}\lambda u&u^2\\u-\lambda^2&-\lambda u\end{pmatrix},
 \qquad
 B_\lambda^{-1}rB_\lambda=
 \begin{pmatrix}0&0\\u^2&0\end{pmatrix}.
\]
Thus every first neighbour is $\Lambda$-stable.  If $\lambda\ne0$, the
reduction modulo $u$ contains $E_{21}$, whose unique invariant line is the
line corresponding to the edge back to $[L_\circ]$.  The same holds for the
neighbour $L_\infty$.

For $\lambda=0$, all non-scalar generators reduce to zero, hence every line
in $L_0/uL_0$ is invariant.  The outgoing neighbours other than the edge
back to $L_\circ$ are precisely the lattices $L_{0,\mu}$.  With
\[
 C_\mu=\begin{pmatrix}1&0\\\mu u&u^2\end{pmatrix}
\]
one computes
\[
 C_\mu^{-1}pC_\mu=
 \begin{pmatrix}0&0\\\mu&u\end{pmatrix},
\]
\[
 C_\mu^{-1}qC_\mu=
 \begin{pmatrix}\mu u^2&u^3\\1-\mu^2u&-\mu u^2\end{pmatrix},
 \qquad
 C_\mu^{-1}rC_\mu=
 \begin{pmatrix}0&0\\u&0\end{pmatrix}.
\]
Again the reduction contains $E_{21}$, and its unique invariant line is the
edge back to $[L_0]$.  Hence there is no third layer.  Since the branch of
an order is a subtree, this exhausts all maximal overorders.
\end{proof}

\begin{corollary}\label{cor:hereditary-overorders}
Every hereditary overorder of $\Lambda$ is either one of the maximal orders
in \eqref{eq:max-overorders}, or the Iwahori order given by the intersection
of the two maximal orders attached to an edge of the branch in
\Cref{thm:BT-branch}.  Conversely every such vertex or edge order contains
$\Lambda$.
\end{corollary}

\begin{proof}
Over a DVR, a hereditary order in $\Mat_2(F)$ is either maximal or an
Iwahori order, the latter being the intersection of the two maximal orders
at the endpoints of an edge.  The assertion therefore follows from
\Cref{thm:BT-branch}.
\end{proof}

\begin{remark}\label{rem:missing-Ringel}
The canonical radical-idealizer chain gives the
nine-dimensional algebra below. A $22/40$-dimensional pair is obtained using an expanded chain whose auxiliary choices
are conjugate under automorphisms of $\Lambda$; see
\Cref{thm:expanded-choice} and \Cref{prop:expanded-quotient,prop:expanded-Ringel}.
Consequently both constructions give isomorphism invariants, but
they give different algebras. More precisely,
$Q_{\mathrm{can}}\cong eQ_{\mathrm{exp}}e/(da)$, where
$e=e_0+e_1$. In particular the relation is a quotient of a corner,
and the $22/40$ construction does not come from the canonical
radical-idealizer chain. We use left modules and define each Ringel dual as the opposite
endomorphism algebra of its characteristic tilting module.
\end{remark}

\subsection{\texorpdfstring{The finite quotient of the radical-idealizer chain}
{The finite quotient of the radical-idealizer chain}}
\begingroup
Throughout this subsection and the next, modules are left modules
and endomorphisms are multiplied by ordinary composition. We put
$R(Q)=\operatorname{End}_Q(T)^{\mathrm{op}}$ for the characteristic
left tilting module $T$. Paths are still multiplied from left to
right: an arrow $i\to j$ in $e_iQe_j$ acts on a left module from
the space at $j$ to the space at $i$. Write
$D_R(L)=\operatorname{Hom}_R(L,R)$ and
$D_\kk(V)=\operatorname{Hom}_\kk(V,\kk)$.
For an overorder $L$, the lattice $D_R(L)$ is a left $\Lambda$-module
via $(a\varphi)(\ell)=\varphi(\ell a)$.

Write $I=\Lambda_2$, $J_i=J(\Lambda_i)$ for $i=0,1$, and set
\[
 \mathcal K_{\mathrm{can}}
 =\operatorname{End}_{\Lambda_0}
   (D_R(\Lambda_0)\oplus D_R(\Lambda_1)\oplus D_R(I))^{\mathrm{op}}.
\]
Let $e_I$ project onto the $D_R(I)$-summand. The following finite algebra is
intrinsic to this specified radical-idealizer chain; it differs from the expanded quotient studied below.

\begin{theorem}\label{thm:corrected-Koenig-quotient}
Identifying the $(i,j)$ entry with
\[
 \operatorname{Hom}_{\Lambda_0}(D_R(\Lambda_i),D_R(\Lambda_j)),
\]
one has
\[
 \mathcal K_{\mathrm{can}}=
 \begin{pmatrix}
  \Lambda_0&J_0&u^2I\\
  \Lambda_1&\Lambda_1&J_1\\
  I&I&I
 \end{pmatrix}.
\]
The quotient $Q_{\mathrm{can}}=\mathcal K_{\mathrm{can}}/
\mathcal K_{\mathrm{can}}e_I\mathcal K_{\mathrm{can}}$ is basic of
dimension nine over $\kk$.  It has vertices $0,1$, arrows
$a_1,a_2,a_3:0\to1$ and $d:1\to0$, and relations
\[
 da_1=da_2=da_3=0.
\]
Paths are multiplied from left to right.  A basis consists of
$e_0,e_1,a_1,a_2,a_3,d,a_1d,a_2d,a_3d$.
It is quasi-hereditary for the order $0<1$.
\end{theorem}
\begin{proof}
Dualizing an entry gives an $R$-linear map
$\Lambda_j\to\Lambda_i$ commuting with multiplication on the right.
After extension to $F=\kk((u))$, it is left multiplication by a
matrix $M$. Since $1\in\Lambda_j$, one has $M\in\Lambda_i$;
the remaining condition is $M\Lambda_j\subseteq\Lambda_i$.
Duality reverses composition, which is compensated by the opposite
endomorphism ring in the definition of $\mathcal K_{\mathrm{can}}$.
Applying this condition to the
matrix-unit generators in \cref{thm:idealizer-chain} gives
\[
 (\Lambda_0:\Lambda_1)=J_0,\qquad
 (\Lambda_0:I)=u^2I
 =\begin{pmatrix}u^2R&u^2R\\u^3R&u^2R\end{pmatrix},
 \qquad (\Lambda_1:I)=J_1.
\]
All other displayed entries follow from the inclusions of orders.
Composing through $I$ gives the entire lattice
$(\Lambda_i:I)$ in each block of row $i$, because
$(I:\Lambda_j)=I$ and $1\in I$.  Consequently the four blocks of
$Q_{\mathrm{can}}$ are
\[
 \begin{pmatrix}
 \Lambda_0/u^2I&J_0/u^2I\\
 \Lambda_1/J_1&\Lambda_1/J_1
 \end{pmatrix}.
\]
The respective $\kk$-lengths are $4,3,1,1$.  Indeed an explicit
basis of $\Lambda_0/u^2I$ is
$1,u1,p,q$, where $p=uE_{22}$ and
$q=uE_{12}+u^2E_{21}$; the last three also form a basis of
$J_0/u^2I$.  The two lower blocks are represented by $1$.
Take $a_1=u1$, $a_2=p$, $a_3=q$ in the $(0,1)$ block and $d=1$
in the $(1,0)$ block.  Then the $a_id$ give the three independent
radical classes in the $(0,0)$ block, whereas each $da_i$ lies in
$J_1$ and hence vanishes in the $(1,1)$ block.  The displayed
relations kill every path of length at least three; the nine listed
paths exhaust the block lengths, proving the presentation.

Finally $e_1Q_{\mathrm{can}}e_1=\kk$, and multiplication
$Q_{\mathrm{can}}e_1\otimes_{\kk}e_1Q_{\mathrm{can}}
\longrightarrow Q_{\mathrm{can}}e_1Q_{\mathrm{can}}$
is an isomorphism: its eight basis tensors map to
$e_1,d,a_i,a_id$.  Thus $Q_{\mathrm{can}}e_1Q_{\mathrm{can}}$
is a heredity ideal and the quotient is $\kk e_0$.
\end{proof}

\begin{theorem}
\label{thm:corrected-Ringel}
For left modules over $Q_{\mathrm{can}}$, the characteristic tilting
module for $0<1$ is $T=T(0)\oplus T(1)$, where $T(0)=S_0$ and
\[
 T(1)_0=\kk^4=\langle v_1,v_2,v_3,w\rangle,
 \qquad T(1)_1=\kk z,\qquad
 a_i(z)=v_i,\qquad d(w)=z,\qquad d(v_i)=0.
\]
Its Ringel dual $\mathcal R_{\mathrm{can}}=
\operatorname{End}_{Q_{\mathrm{can}}}(T)^{\mathrm{op}}$ has dimension nine.
With the same vertex labels, it has arrows $h:1\to0$ and
$f_1,f_2,f_3:0\to1$, and relations $f_ih=0$ for $i=1,2,3$.
Consequently $\mathcal R_{\mathrm{can}}\cong Q_{\mathrm{can}}^{\mathrm{op}}$
after swapping vertices $0$ and $1$; its highest-weight order is
reversed. In particular $Q_{\mathrm{can}}$ is not Ringel self-dual.
\end{theorem}
\begin{proof}
For $0<1$, $\Delta(0)=\nabla(0)=S_0$,
$\Delta(1)=P_1=Q_{\mathrm{can}}e_1$, and
$\nabla(1)=I_1=D_\kk(e_1Q_{\mathrm{can}})$.
The displayed maps satisfy $da_i=0$, since on left modules this
path acts as $d\circ a_i$. The subrepresentation spanned by
$v_1,v_2,v_3,z$ is $P_1$, with quotient $S_0$.
The subrepresentation $\langle v_1,v_2,v_3\rangle$ is
$S_0^{\oplus3}$, with quotient $I_1$ on $w,z$.
Thus $T(1)$ has both required filtrations. Its endomorphism ring
has scalar diagonal and three maps $w\mapsto v_i$, so it is local;
$T(1)$ is indecomposable.
This identifies the characteristic tilting module.

There is one independent map $h:T(1)\to T(0)$, taking $w$ to $1$
and killing $v_i,z$, and three maps $f_i:T(0)\to T(1)$ with
images $\kk v_i$. Hence $h\circ f_i=0$, while
the three composites $f_i\circ h$ are linearly independent nilpotent
endomorphisms of $T(1)$.  The four Hom-space dimensions between
the two tilting summands are $1,3,1,4$, summing to nine.
In $\operatorname{End}(T)^{\mathrm{op}}$, path multiplication is
composition in its displayed order, so $f_ih$ in the quiver
represents $h\circ f_i=0$, and $hf_i$ represents
$f_i\circ h\ne0$. The nine surviving paths give all nine
endomorphisms. Swapping the vertices identifies this quiver with
the opposite of that in \cref{thm:corrected-Koenig-quotient}, so
$\mathcal R_{\mathrm{can}}^{\mathrm{op}}\cong Q_{\mathrm{can}}$.
The latter is not isomorphic to its opposite: its primitive corner
dimensions $4,1$ force an isomorphism to preserve the vertices,
whereas the arrow-space dimensions $3,1$ are interchanged by
passing to the opposite algebra.
\end{proof}

\subsection{The expanded chain and its Ringel dual}

The following construction gives a second finite quotient with a
precise invariance statement. It does not require the expanded endomorphism
order to be a Koenig resolution.

\begin{theorem}
\label{thm:expanded-choice}
Let $I$ be the hereditary terminal order of the canonical chain.
Choose a maximal order $\Gamma'$ containing $I$, and a semisimple
complement $S\cong\kk\times\kk$ to the radical of $I/u\Gamma'$.
Let $B$ be the inverse image of $S$ in $I$. Set
\[
 \mathcal E=\operatorname{End}_{\Lambda}
 (D_R(\Lambda)\oplus D_R(\Lambda_1)\oplus D_R(B)\oplus D_R(\Gamma'))^{\mathrm{op}},\qquad
 Q_{\mathrm{exp}}=\mathcal E/\mathcal E e_{\Gamma'}\mathcal E.
\]
Here $e_{\Gamma'}$ projects onto $D_R(\Gamma')$.
The isomorphism class of $Q_{\mathrm{exp}}$, with weights
$0<1<\{+,-\}$ attached to $\Lambda,\Lambda_1,B$, is independent
of both choices. Thus it is an isomorphism invariant of $A_\infty$.
\end{theorem}
\begin{proof}
For $\Gamma=\Mat_2(R)$, the algebra $I/u\Gamma$ is the upper
triangular $2\times2$ algebra over $\kk$. Its semisimple complements
are the conjugates of the diagonal algebra by
$\left(\begin{smallmatrix}1&t\\0&1\end{smallmatrix}\right)$,
$t\in\kk$. Their inverse images are the orders $B_t$; at $t=0$,
\[
 B_0=\begin{pmatrix}R&uR\\uR&R\end{pmatrix}.
\]
Here $u\Gamma$ is a two-sided ideal of $I$. The image of
$\Lambda_1$ in $I/u\Gamma$ is the scalar subalgebra, which belongs
to every unital semisimple complement. Thus
$\Lambda\subset\Lambda_1\subset B_t\subset I$, so each chosen
$D_R(B_t)$ is indeed a left $\Lambda$-module. Replacing $u$ by any
other uniformizer does not change the ideal $u\Gamma$.
All these choices are conjugate by automorphisms of $\Lambda$.
Indeed put
\[
 g_t=\begin{pmatrix}1&t\\-ut&1\end{pmatrix},\qquad D=1+ut^2.
\]
In the basis $1,p,q,r$ of \eqref{eq:Lambda-basis-matrix}, direct
matrix multiplication gives
\[
\begin{aligned}
 g_tpg_t^{-1}&=D^{-1}\bigl(u^2t^2\,1+(1-ut^2)p+tq\bigr),\\
 g_tqg_t^{-1}&=D^{-1}\bigl(2u^2t\,1-4utp+(1-ut^2)q\bigr),\\
 g_trg_t^{-1}&=D^{-1}\bigl(u^3t\,1-2u^2tp-u^2t^2q\bigr)+r.
\end{aligned}
\]
As $D$ is a unit and $g_t^{-1}=D^{-1}g_{-t}$, these identities
prove $g_t\Lambda g_t^{-1}=\Lambda$. Also $g_t\Gamma g_t^{-1}=\Gamma$
and $g_tB_0g_t^{-1}=B_t$.

There are exactly two maximal overorders of $I$, namely
$\Gamma$ and $\Gamma^{\dagger}=
\left(\begin{smallmatrix}R&u^{-1}R\\uR&R\end{smallmatrix}\right)$.
The matrix
\[
 W=\begin{pmatrix}0&1\\u&0\end{pmatrix}
\]
interchanges them, preserves $I$, and normalizes $\Lambda$, since
\[
 WpW^{-1}=u1-p,\qquad WqW^{-1}=q,\qquad WrW^{-1}=uq-r.
\]
It therefore reduces the second maximal-order choice to the first.
Conjugation transports the entire module sum, its endomorphism
algebra, and the ideal through $\Gamma'$. It fixes the first two
weights and at most permutes the two incomparable top weights.
This proves independence. Finally, normalization of the centre,
the torsion-free normalized order, and its radical-idealizer chain
are preserved by algebra isomorphisms; hence the construction
depends only on the isomorphism class of $A_\infty$.
\end{proof}

\begin{proposition}
\label{prop:expanded-quotient}
The algebra $Q_{\mathrm{exp}}$ is a $22$-dimensional algebra:
it has vertices $0,1,+,-$, arrows
\[
 a,b:0\to1,\quad c:0\to+,\quad d:1\to0,\quad
 e:1\to-,\quad f:+\to1,\quad g:-\to1,
\]
and relations, with paths multiplied from left to right,
\[
 be=db=dc=fe=ge=0,\qquad da=eg.
\]
Writing $\epsilon=e_0+e_1$, one has
\[
 \dim_\kk\epsilon Q_{\mathrm{exp}}\epsilon=13,\qquad
 Q_{\mathrm{can}}\cong
 \epsilon Q_{\mathrm{exp}}\epsilon/(da),
\]
where $(da)$ is the ideal inside the corner and has basis
$eg,aeg,egd,aegd$.
\end{proposition}
\begin{proof}
Use $B=B_0$ and $\Gamma'=\Gamma$. The requisite additional colons are
\[
 (\Lambda:B)=\begin{pmatrix}uR&u^2R\\u^3R&u^2R\end{pmatrix},\quad
 (\Lambda:\Gamma)=\begin{pmatrix}u^2R&u^2R\\u^3R&u^3R\end{pmatrix},
\]
\[
 (\Lambda_1:B)=J_1,\qquad
 (\Lambda_1:\Gamma)=\begin{pmatrix}uR&uR\\u^2R&u^2R\end{pmatrix},
 \qquad (B:\Gamma)=u\Gamma.
\]
They follow by testing multiplication by the four matrix-unit
generators of $B$ or $\Gamma$. Composing through $\Gamma$ removes
$(L:\Gamma)$ from every block with target $L$.
After splitting $B$ by $E_{11},E_{22}$, the block dimensions are
\[
 \begin{pmatrix}5&4&1&1\\2&2&0&1\\1&1&1&0\\1&1&0&1\end{pmatrix}.
\]
Represent the arrows by
\[
 a=u1,\ b=q,\ c=uE_{11},\ d=1,\ e=uE_{22},\ f=E_{11},\ g=E_{22}
\]
in their respective blocks. Multiplication verifies all six
relations. Besides the four idempotents and seven arrows, the
normal paths are
\[
 ad,ae,bd,cf,eg,fd,gd;\qquad aeg,cfd,egd;\qquad aegd.
\]
Their matrix classes form bases of the displayed blocks, proving
the presentation and dimension. The corner has basis
\[
 e_0,e_1,a,b,d,ad,bd,cf,eg,aeg,cfd,egd,aegd.
\]
Its ideal generated by $da=eg$ is exactly the four-dimensional
span stated above. Modulo that ideal, the three arrows $a,b,cf$
from $0$ to $1$ and the arrow $d$ satisfy
$da=db=d(cf)=0$, and their three composites with $d$ are
independent. This is the nine-dimensional presentation in
\cref{thm:corrected-Koenig-quotient}. The identification also
respects the lattice construction: the ideals of maps through
$\Gamma$ in the corner are contained in the ideals of maps
through $I$, and their four block quotients have length one each.
\end{proof}

\begin{proposition}
\label{prop:expanded-Ringel}
For $0<1<\{+,-\}$, the algebra $Q_{\mathrm{exp}}$ is
quasi-hereditary and its Ringel dual has dimension $40$.
The ordinary endomorphism algebra is the opposite of this
Ringel dual. Both isomorphism classes are invariants of
$A_\infty$ under the expanded construction above.
Its multiplication is specified by the characteristic tilting
module in the proof; the opposite algebra uses ordinary composition.
\end{proposition}
\begin{proof}
The top heredity ideal generated by $e_++e_-$ has dimension $15$;
the multiplication map from the tensor product of the corresponding
projective bimodules is an isomorphism, as is checked in the normal
path basis above. Its quotient has basis
$e_0,e_1,a,b,d,ad,bd$. In this quotient the ideal generated by
$e_1$ is a six-dimensional heredity ideal, with quotient $\kk e_0$.
Thus the indicated weight order is quasi-hereditary.

Here are explicit characteristic tilting summands. Unspecified
spaces and maps are zero. Take $T_0=S_0$. For $T_1$ take bases
$(u_0,u_a,u_b)$ at $0$ and $v$ at $1$, with
$d(u_0)=v$, $a(v)=u_a$, $b(v)=u_b$.
For $T_+$ add $u_c$ at $0$ and $w$ at $+$, with
$c(w)=u_c$ and $f(v)=w$. For $T_-$ take bases
$(p_0,q_0,r_0,s_0)$ at $0$, $(p_1,q_1)$ at $1$, and $w$ at $-$,
with
\[
\begin{gathered}
 d(p_0)=p_1,\quad d(q_0)=q_1,\quad g(p_1)=w,\quad e(w)=q_1,\\
 a(p_1)=q_0,\quad a(q_1)=r_0,\quad b(p_1)=s_0.
\end{gathered}
\]
The standard modules have dimension vectors
\[
 (1,0,0,0),\quad(2,1,0,0),\quad(1,0,1,0),\quad(1,1,0,1),
\]
with nonzero maps, respectively: none; the two inclusions $a,b$;
$c$; and $e,a$ acting consecutively from $-$ to $1$ to $0$.
The costandard modules have dimension vectors
\[
 (1,0,0,0),\quad(1,1,0,0),\quad(1,1,1,0),\quad(1,1,0,1),
\]
with nonzero maps, respectively: none; $d$; $d,f$; and $d,g$.
These follow from the projectives $Q_{\mathrm{exp}}e_i$ and
injectives $D_\kk(e_iQ_{\mathrm{exp}})$ by the weight truncations.

In $T_1$ the submodule spanned by $u_a,u_b,v$ is $\Delta(1)$,
with quotient $S_0$. The submodules $\langle u_c,w\rangle$ of
$T_+$ and $\langle r_0,q_1,w\rangle$ of $T_-$ are
$\Delta(+)$ and $\Delta(-)$, respectively, with quotient $T_1$.
These give standard filtrations. For costandard filtrations,
$\langle u_a,u_b\rangle\subset T_1$ has quotient $\nabla(1)$,
and $\langle u_a,u_b,u_c\rangle\subset T_+$ has quotient
$\nabla(+)$. In $T_-$ first take $\langle r_0,s_0\rangle$
and then adjoin $q_0,q_1$; the successive factors are
$S_0^{\oplus2},\nabla(1),\nabla(-)$.
The endomorphism ring of each summand is local: its common scalar
diagonal has a nilpotent kernel. Hence these are the four
indecomposable characteristic tilting summands.

For a quasi-hereditary algebra, standard--costandard
orthogonality gives
\[
 \dim\operatorname{Hom}(M,N)
 =\sum_\lambda[M:\Delta(\lambda)][N:\nabla(\lambda)]
\]
whenever $M$ has a standard filtration and $N$ has a costandard
filtration. This follows by induction on both filtrations from
$\operatorname{Hom}(\Delta(\lambda),\nabla(\mu))
=\delta_{\lambda\mu}\kk$ and
$\operatorname{Ext}^1(\Delta(\lambda),\nabla(\mu))=0$
\cite{Ringel1991}.
The filtrations above give the multiplicity matrices
\[
 D_\Delta=\begin{pmatrix}1&0&0&0\\1&1&0&0\\1&1&1&0\\1&1&0&1\end{pmatrix},
 \qquad
 D_\nabla=\begin{pmatrix}1&0&0&0\\2&1&0&0\\3&0&1&0\\2&1&0&1\end{pmatrix}.
\]
Hence
\[
 \bigl(\dim\operatorname{Hom}(T_i,T_j)\bigr)_{i,j}
 =D_\Delta D_\nabla^{\mathsf T}
 =\begin{pmatrix}1&2&3&2\\1&3&3&3\\1&3&4&3\\1&3&3&4\end{pmatrix}.
\]
The entries sum to $40$, proving the dimension of the Ringel
dual without solving intertwining equations or enumerating paths.
Its opposite is the ordinary endomorphism algebra of the four
explicit tilting summands; this intrinsic description fixes its
multiplication and avoids a second quiver convention.
Finally \Cref{thm:expanded-choice}
preserves the weight order, so it also preserves the characteristic
tilting object, its endomorphism algebra, and its opposite algebra.
\end{proof}

\begin{remark}
The expanded and canonical quotient/Ringel-dual pairs have dimensions
$22/40$ and $9/9$, respectively. Taking opposite algebras does not
change these dimensions. The former have four simple modules and the latter
two. Hence no algebra in the expanded pair is Morita equivalent
or bounded-derived equivalent to either algebra in the canonical
pair: their Grothendieck groups have ranks four and two.
Within the expanded pair, however, $Q_{\mathrm{exp}}$ and its
Ringel dual are derived equivalent via the characteristic tilting
module. They are not Morita equivalent, since they are basic and
have different dimensions. The canonical Ringel dual is isomorphic to
$Q_{\mathrm{can}}^{\mathrm{op}}$, rather than $Q_{\mathrm{can}}$.
The ordinary tilting endomorphism algebra, being the opposite of
$R(Q_{\mathrm{exp}})$, is derived equivalent to
$Q_{\mathrm{exp}}^{\mathrm{op}}$; its derived equivalence to
$Q_{\mathrm{exp}}$ is not asserted here.
The two constructions are related by the corner quotient described
in \Cref{prop:expanded-quotient}.
The invariance statements follow from
\Cref{thm:expanded-choice} and \Cref{prop:expanded-Ringel};
in particular, they do not require a global-dimension bound for
$\mathcal E$ or an identification of $\mathcal E$ with the
endomorphism order attached to the canonical radical-idealizer chain.
\end{remark}
\endgroup

\section{A finite-to-infinite bridge}\label{sec:finite-infinite}

The rank twelve of $A_\infty$ over $\kk[[z]]$ explains the recurrent
increment by twelve in the finite $E_{6,n}$ family.  There is also an
explicit formal bridge.  For $n\ge5$ set
\begin{equation*}
 \mathfrak A_n=
 \frac{\kk[[s]]\langle\!\langle x,y\rangle\!\rangle}
 {(x^2-y^3,\;xy^2+yxy+y^2x+s^{2n-9}y^{n-1})}.
\end{equation*}
Then
\[
 \mathfrak A_n/(s)\cong A_\infty.
\]
After inverting $s$, the change of variables
\[
 X=s^3x,\qquad Y=s^2y
\]
identifies the augmentation-adic completion of the generic fibre
with the algebra defined by the same
two normalized relations as the finite $E_{6,n}$ form; a further
constant rescaling of $X,Y$ identifies this with $A_n$ over
$\kk((s))$. Indeed
\[
 X^2-Y^3=s^6(x^2-y^3)
\]
and
\[
 XY^2+YXY+Y^2X+Y^{n-1}
 =s^7\bigl(xy^2+yxy+y^2x+s^{2n-9}y^{n-1}\bigr).
\]
Thus $E_{6,\infty}$ is a punctured-isotrivial limit in this completed sense of every
$E_{6,n}$ with $n\ge5$.  The identity
\[
 \dim A_{n+3}-\dim A_n=12
 =\operatorname{rank}_{\kk[[z]]}A_\infty
\]
then acquires a structural interpretation.

\appendix
\section{Uniform terminal cyclic identities and the centre}\label{app:center}
Brown--Wemyss list the higher coranks
\(\operatorname{Crk}_i=\dim_\kk J(A_n)^i/J(A_n)^{i+1}\)
through \(i=n+6\) and give total length \(4n+12\)
\cite[\S1.3 and Definition~4.9]{BW2025}. The listed dimensions,
including \(\operatorname{Crk}_0=1\) and \(\operatorname{Crk}_1=2\),
already sum to \(4n+12\). Thus all later quotients vanish and,
by nilpotence of the radical of this finite-dimensional algebra,
\begin{equation}\label{eq:cyc-nil}
 J(A_n)^{n+7}=0\qquad(n\ge5).
\end{equation}
The order-four endpoint is proved by terminal relations in \cref{app:four}.
Throughout this appendix $[a]$ denotes the class of $a$ in the
\emph{vector-space} quotient $\HH_0(A_n)=A_n/[A_n,A_n]$;
the commutator subspace is not treated as a two-sided ideal.
In \(A_n\) put \(z=y^3=-3x^2\), a central element, and write \(n-1=3q+r\) with \(r=0,1,2\) and \(u=nz^q\). The relations give
\begin{equation}\label{eq:cyc-rules}
x^2=-z/3,\quad y^3=z,\quad
xy^2=-yxy-y^2x-uy^r,\quad
xyxy=yxyx-u(xy^r-y^rx).
\end{equation}
The last relation follows by commuting \(x\) with the second defining relation \eqref{eq:r2}. If \(r=2\), its right side, after one application of the third rule, is \(yxyx+u(yxy+2y^2x)+u^2y^2\).
\begin{lemma}\label{lem:cyc-rho}
For the appropriate \(r\), the following element is zero in the algebra \(A_n\):
\[
\begin{array}{c|l}
r&\rho_r\\ \hline
0&u(xyx+zy/3)\\
1&u\,yxyx+u^2yx-\tfrac23uz\,y^2\\
2&u^3z-\tfrac13uz^2+2u^2zx+2u^2y^2xy+2u\,y^2xyx.
\end{array}
\]
\end{lemma}
\begin{proof}
Reduce the overlap word \(xyxxyy\) in two ways: apply the rule for \(x^2\) to its middle \(xx\), or apply the rule for \(xy^2\) to its final \(xyy\). Reduce the resulting words again by the same four rules. The difference between the two paths is respectively
\[
u\,xyx+\tfrac13uz\,y,\qquad
u\,yxyx+u^2yx-\tfrac23uz\,y^2,
\]
\[
2u^3z-\tfrac23uz^2+4u^2zx+4u^2y^2xy+4u\,y^2xyx.
\]
These expressions are \(\rho_0,\rho_1,2\rho_2\), respectively. Both paths represent the same original word, so the differences vanish. All rules strictly lower either word length or length-lexicographic order \(x>y\); this finite computation does not use confluence or a numerically guessed pattern.
\end{proof}
Two cyclic reductions suffice for the trace calculation. For \(j\ge2\), rotating the second relation \eqref{eq:r2} gives
\begin{equation}\label{eq:cyc-one}
[y^jx]=-\frac n3[y^{n+j-3}].
\end{equation}
Put \(T_i=[y^ixyx]\), \(i=0,1,2\). Rotating the multiples \(r_yx\) and \(r_yxy\), then using \(x^2=-y^3/3\) and \eqref{eq:cyc-one}, gives
\begin{equation}\label{eq:cyc-two}
T_0=-\tfrac13[y^4],\qquad
T_1=\tfrac23[y^5]+\tfrac{n^2}{3}[y^{2n-4}],\qquad
T_2=\tfrac16[y^6]+\tfrac{n^2}{6}[y^{2n-3}].
\end{equation}
For clarity, \(r_yx\) has two cyclic summands equal to \(-[y^5]/3\); \(r_yxy\) has two equal summands \(T_2\). Their last terms are evaluated by \eqref{eq:cyc-one}.
\begin{proposition}\label{prop:cyc-terminal}
For every \(n\ge4\),
\begin{equation}\label{eq:cyc-terminal}
\boxed{[y^{n+4}]+\frac{n^2}{3}[y^{3n-5}]=0,\qquad
[y^{n+5}]+\frac{n^2}{3}[y^{3n-4}]=0.}
\end{equation}
Consequently, for every \(n\ge6\), \([y^{n+4}]=[y^{n+5}]=0\). For \(n=5\) the exact exceptional relations are
\([y^9]+\frac{25}{3}[y^{10}]=0\) and
\([y^{10}]+\frac{25}{3}[y^{11}]=0\).
\end{proposition}
\begin{proof}
Multiply the algebra identity \(\rho_r=0\) first by \(y\), then by \(y^2\), before passing to the cyclic quotient. Substitution of \eqref{eq:cyc-one} and \eqref{eq:cyc-two}, with \(y^3=z\), yields the complete coefficient table
\[
\begin{array}{c|cc}
r&[y\rho_r]&[y^2\rho_r]\\ \hline
0&nU&(n/2)V\\
1&-(n/2)U&-nV\\
2&-nU&nV
\end{array}
\quad
\begin{aligned}
U&=[y^{n+4}]+\tfrac{n^2}{3}[y^{3n-5}],\\
V&=[y^{n+5}]+\tfrac{n^2}{3}[y^{3n-4}].
\end{aligned}
\]
For example \([y\rho_0]=u(T_1+\frac z3[y^2])=nU\); all entries follow by the same displayed trace reductions. Every table coefficient is nonzero in characteristic zero. When \(n\ge6\), the exponents \(3n-5,3n-4\) exceed \(n+6\), so the corresponding elements vanish in \(A_n\) itself by \eqref{eq:cyc-nil}.
\end{proof}

\subsection{The cyclic spanning bound}
\begin{lemma}\label{lem:cyc-span}
The algebra \(A_n\) is spanned over the central subring \(\kk[z]\) by the twelve words
\begin{equation}\label{eq:cyc-twelve}
y^i,\quad y^ix,\quad y^ixy,\quad y^ixyx,\qquad 0\le i\le2.
\end{equation}
Consequently \(\HH_0(A_n)\) is spanned over \(\kk\) by \([y^j]\) (\(j\ge0\)), \([x]\), and \([xy]\).
\end{lemma}
\begin{proof}
Use \eqref{eq:cyc-rules}, with its indicated reduction for \(r=2\). Each substitution lowers the word in length-lex order; the words avoiding \(x^2,y^3,xy^2,xyxy\) are exactly \eqref{eq:cyc-twelve}. Confluence is unnecessary for spanning. Cyclic rotation and \eqref{eq:cyc-one} reduce \(y^ix\) for \(i=2\); \(y^ixy\) rotates to \(xy^{i+1}\); and the three \(y^ixyx\) reduce by \eqref{eq:cyc-two}. The cases \(i=0,1\) of \(y^ix\) give \([x]\) and \([xy]\).
\end{proof}
For \(n\ge5\), \eqref{eq:cyc-nil} leaves at most \(n+7\) pure-power classes \([1],\ldots,[y^{n+6}]\) and two additional classes. The relations \eqref{eq:cyc-terminal} are linearly independent as relations on this spanning list: if \(n\ge6\), they eliminate the distinct columns \(n+4,n+5\); if \(n=5\), their restrictions to columns \(9,10,11\) are independent. Thus
\begin{equation*}
\dim_{\kk}\HH_0(A_n)\le n+7\qquad(n\ge5).
\end{equation*}
At \(n=4\), Appendix~\ref{app:four} gives \(y^{10}=0\), so there are initially twelve candidate classes. The two independent equations \eqref{eq:cyc-terminal} reduce the bound to ten.

\subsection{Saturation, explicit cyclic bases, and rank}
\begin{theorem}
\label{thm:cyc-rank}
For every $n\ge5$ one has
\[
 \dim_\kk\HH_0(A_n)=\dim_\kk Z(A_n)=n+7,
 \qquad \operatorname{rank}\Gamma_n=3n+5,
\]
so $\operatorname{rank}\Gamma_{n+1}-\operatorname{rank}\Gamma_n=3$.
For the endpoint, $\dim A_4=27$, $\dim Z(A_4)=10$,
and $\operatorname{rank}\Gamma_4=17$.
\end{theorem}
\begin{proof}
By \Cref{thm:main-existence}, the completed deformation
$\mathcal A_n$ is free over $\kk[[t]]$ and has semisimple geometric
generic fibre $S_n$.  The rank of the commutator map cannot increase
under specialization from the generic fibre to $A_n$; equivalently,
\[
 \dim\HH_0(A_n)\ge\dim\HH_0(S_n)
 =\dim Z(S_n)=
 \begin{cases}10,&n=4,\\ n+7,&n\ge5.\end{cases}
\]
This lower bound saturates the cyclic spanning upper bound proved above.
Explicit bases of the \emph{vector-space} quotient are
\[
\begin{array}{c|l}
n&\text{basis of }\HH_0(A_n)\\ \hline
4&[1],[y],\ldots,[y^7],\ [x],[xy]\\
5&[1],[y],\ldots,[y^8],\ [y^{11}],\ [x],[xy]\\
n\ge6&[1],[y],\ldots,[y^{n+3}],\ [y^{n+6}],\ [x],[xy].
\end{array}
\]
For $n=5$ the two relations in \eqref{eq:cyc-terminal} express
$[y^9]$ and $[y^{10}]$ in terms of $[y^{11}]$; for $n\ge6$
they remove $[y^{n+4}]$ and $[y^{n+5}]$.  For $n=4$ they remove
$[y^8]$ and $[y^9]$ using $[y^7]$.
Since $A_n$ is symmetric, its symmetrizing form identifies
$Z(A_n)$ with the dual of $\HH_0(A_n)$.
Finally $\ker\Gamma_n=Z(A_n)$; subtracting the centre dimension
from $\dim A_n=4n+12$ for $n\ge5$, or $27$ for $n=4$,
gives the displayed ranks and the increment.
\end{proof}

\section{The order-four endpoint by terminal relations}\label{app:four}
The special case of \cref{lem:order-four-upper-bound} admits
exact terminal relations, rather than only initial forms.

\begin{proposition}\label{prop:four-terminal}
For $p(y)=cy^4$, $c\ne0$, put $A=4c$ and $z=y^3$.
In $A_p$ one has
\[
\begin{gathered}
 xyxy=yxyx,\qquad zxyx=-zy^4/3,\qquad
 z^2xy=z^2yx,\\
 y^8x=-Ay^9/3,\qquad y^{10}=0.
\end{gathered}
\]
Consequently $\dim A_p\le27$ and $J^{10}=0$.
For $c=1$, the semisimple quotient constructed in
\cref{prop:scalar,prop:two,prop:three} saturates this bound.
\end{proposition}
\begin{proof}
Commuting the second Jacobi relation with $x$ gives the first
identity. The overlap identity is
$Az(xyx+zy/3)=0$, giving the second. Its commutator with $x$
gives the third. Multiplying the second Jacobi relation by $z^2$
and using the third gives $3y^8x+Ay^9=0$.
Multiplication by $x$ and by $y$, respectively, gives
\[
 y^{11}=Ay^9x,\qquad 3y^9x+Ay^{10}=0.
\]
Therefore $y^{10}(A^2+3y)=0$. The second factor is a unit in
the completed local algebra, so $y^{10}=0$.
The initial forms of these identities and the Jacobi relations
are the seven homogeneous reductions in
\Cref{lem:order-four-upper-bound}. They give the same $27$
spanning words and no words in the associated graded algebra of
degree ten or higher. This proves the bounds. The generic
semisimple quotient has dimension
$6+3\cdot4+9=27$, so \Cref{lem:topNak} proves equality.
\end{proof}

\end{document}